\documentclass{amsart}
\usepackage[top=1 in,bottom=1 in, left=1.5 in, right= 1.5 in, includehead,includefoot]{geometry}

          \usepackage{amssymb}
          \usepackage{amsmath}
          \usepackage{amsthm}
          \usepackage{amsfonts}
          \usepackage[english]{babel}
          \usepackage[utf8]{inputenc}
          \usepackage{enumerate}
          \usepackage{graphicx}
          \usepackage{url}
          \usepackage{color}

\usepackage[numbers,sort&compress]{natbib}
\newcommand{\comment}[1]{}

\newcommand{\ind}{{\bf 1}}
\def\indd#1{{\ind}_{\{#1\}}}
\def\inddd#1{{\ind}_{\left\{#1\right\}}}
\def\indn#1{\{#1_n\}_{n\in\N}}
\newcommand{\proba}{\mathbb P}
\renewcommand{\P}{\mathbb P}
\newcommand{\esp}{{\mathbb E}}

\newcommand{\var}{{\rm{Var}}}

\newcommand{\eqnh}{\begin{eqnarray*}}
\newcommand{\eqne}{\end{eqnarray*}}
\newcommand{\eqnhn}{\begin{eqnarray}}
\newcommand{\eqnen}{\end{eqnarray}}
\newcommand{\equh}{\begin{equation}}
\newcommand{\eque}{\end{equation}}

\def\summ#1#2#3{\sum_{#1 = #2}^{#3}}
\def\prodd#1#2#3{\prod_{#1 = #2}^{#3}}
\def\sif#1#2{\sum_{#1=#2}^\infty}
\def\bveee#1#2#3{\bigvee_{#1=#2}^{#3}}

\def\bweee#1#2#3{\bigwedge_{#1=#2}^{#3}}

\newcommand{\eqd}{\stackrel{d}{=}}

\def\topp#1{^{(#1)}}

\def\ccbb#1{\left\{#1\right\}}
\def\sccbb#1{\{#1\}}
\def\pp#1{\left(#1\right)}
\def\spp#1{(#1)}
\def\bb#1{\left[#1\right]}

\def\floor#1{\left\lfloor #1 \right\rfloor}

\def\vv#1{{\boldsymbol #1}}

\def\qmand{\quad\mbox{ and }\quad}

\def\qmwith{\quad\mbox{ with }\quad}
\def\mfa{\mbox{ for all }}

\def\mmas{\mbox{ as }}

\def\wt#1{\widetilde{#1}}
\def\wb#1{\overline{#1}}
\def\what#1{\widehat{#1}}

\def\limn{\lim_{n\to\infty}}

\def\limsupn{\limsup_{n\to\infty}}
\def\liminfn{\liminf_{n\to\infty}}
\def\weakto{\Rightarrow}

\def\R{{\mathbb R}}

\def\N{{\mathbb N}}

\def\calF{\mathcal F}
\def\calG{\mathcal G}
\def\calK{\mathcal K}
\def\calM{\mathcal M}
\def\calN{\mathcal N}

\def\calR{\mathcal R}

\def\M{{\mathbb M}}

\def\mab{\calM_{\alpha,\beta}}
\def\Leb{{\rm Leb}}
\def\dH{d_{\rm H}}
\def\sm{{\rm SM}}

\def\vvA{{\vv A}}
\def\vvdelta{{\vv\delta}}

\usepackage{epsfig}
\usepackage{ifthen}

\newtheorem{Thm}{Theorem}[section]
\newtheorem{Lem}[Thm]{Lemma}
\newtheorem{Prop}[Thm]{Proposition}

\theoremstyle{definition}
\newtheorem{Rem}[Thm]{Remark}

\newtheorem{Example}[Thm]{Example}

\numberwithin{equation}{section}

\title[Karlin random sup-measures]{A family of random sup-measures \\with long-range dependence}
\date{\today}

\author{Olivier Durieu}
\address{
Olivier Durieu\\
Institut Denis Poisson, UMR-CNRS 7013\\
Universit\'e de Tours, Parc de Grandmont, 37200 Tours, France.
}
\email{olivier.durieu@univ-tours.fr}

\author{Yizao Wang}
\address
{
Yizao Wang\\
Department of Mathematical Sciences\\
University of Cincinnati\\
2815 Commons Way\\
Cincinnati, OH, 45221-0025, USA.
}
\email{yizao.wang@uc.edu}

\begin{document}\sloppy

\begin{abstract}
A family of self-similar and translation-invariant random sup-measures with long-range dependence are investigated. They are shown to arise as the limit of the empirical random sup-measure of a stationary heavy-tailed process, inspired by an infinite urn scheme, where same values are repeated at several random locations.
The random sup-measure reflects the long-range dependence nature of the original process, and in particular characterizes how locations of extremes appear as long-range clusters represented by random closed sets. A limit theorem for the corresponding point-process convergence is established. 
\end{abstract}

\maketitle

\section{Introduction}
There is a recently renewed interest in limit theorems for extreme values of 
stationary processes in the presence of long-range dependence \citep{beran13long,pipiras17long,samorodnitsky16stochastic}. Let $\{X_n\}_{n\in\N}$ be a stationary process. In extreme value theory, it is by now a classical problem to investigate the limit of the partial maxima $\{\max_{i=1,\dots,\floor{nt}}X_i\}_{t\in[0,1]}$ as a process of $t\in[0,1]$, after appropriate normalization, as $n\to\infty$. It is further understood that such functional limit theorems are better illustrated in terms of convergence of point processes, in particular in the case when the dependence of the extremes of $\{X_n\}_{n\in\N}$ is weak. For a simple and yet representative example, take $\indn X$ to be i.i.d.~heavy-tailed random variables such that $\proba(X_1>x)\sim x^{-\alpha}$ as $x\to\infty$ with tail index $\alpha\in(0,\infty)$.
It is well known that 
\equh\label{eq:independent}
\summ i1n \delta_{(X_i/n^{1/\alpha},i/n)} \weakto \sif \ell1 \delta_{(\Gamma_\ell^{-1/\alpha},U_\ell)},
\eque
where $\{(\Gamma_\ell,U_\ell)\}_{\ell\in\N}$
is a measurable enumeration of points from a Poisson point process on $\R_+\times[0,1]$ with intensity $dxdu$. Such a point-process convergence provides a detailed description of the asymptotic behavior of extremes, by which we mean broadly the top order statistics instead of the largest one alone: the top order statistics normalized by $n^{1/\alpha}$ converge weakly  to $\Gamma_1^{-1/\alpha},\Gamma_2^{-1/\alpha},\dots$, and their locations are asymptotically independent and uniformly distributed over $[0,1]$ \citep{lepage81convergence}. Such a picture is representative for the general situation where $\indn X$ have weak dependence. Classical references now include \citep{resnick87extreme,leadbetter83extremes,dehaan06extreme}, among others.

The recent advances along this line, however, focus on the case when the stationary process $\indn X$ has long-range dependence in the literature. The long-range dependence here,  roughly speaking, means that with the same marginal law, the normalization of maxima is of a different order from $n^{1/\alpha}$ so that a non-degenerate limit arises \citep{samorodnitsky04extreme,samorodnitsky16stochastic}.  In the seminal work of \citet{obrien90stationary}, summarizing a series of developments in the 80s, it has been pointed out that all possible non-degenerate limits of extremes of a stationary sequence can be fit into the framework of convergence of {\em random sup-measures}. The framework could be viewed as a counterpart of the Lamperti's theorem \citep{lamperti62semi} for extremes, in the sense that the limit random sup-measures are necessarily shift-invariant and self-similar. This framework of course includes the case \eqref{eq:independent}, and the corresponding limit random sup-measure on $[0,1]$
can be represented as
\equh\label{eq:completely_random}
\calM_\alpha(\cdot) = \sup_{\ell\in\N}\frac1{\Gamma_\ell^{1/\alpha}}\inddd{U_\ell\in\,\cdot\,},
\eque
or more generally as a random sup-measure on $\R$ in the same notation with $\{(\Gamma_\ell,U_\ell)\}_{\ell\in\N}$ a Poisson point process on $\R_+\times\R$ with intensity $dxdu$. 
In this case, furthermore, the limit random sup-measure is {\em independently scattered} (a.k.a.~completely random) and $\alpha$-Fr\'echet, that is, its values over disjoint sets are independent and for every bounded open set $A$, $\calM_\alpha(A)$ is $\alpha$-Fr\'echet distributed with $\proba(\calM_\alpha(A)\le x) = \exp(-\Leb(A)x^{-\alpha})$, $x>0$. Independently scattered random sup-measures are fundamental in stochastic extremal integral representations of max-stable processes \citep{stoev06extremal}. In general, the random sup-measure arising from a stationary sequence may not be independently scattered, or even Fr\'echet \citep{samorodnitsky17extremal}.

However, within the general framework of convergence of random sup-measures, to the best of our knowledge it is only very recently that other concrete non-trivial examples have been completely worked out.  In a series of papers \citep{owada15maxima,lacaux16time,samorodnitsky17extremal}, the extremes of a well-known challenging example of heavy-tailed stationary processes with long-range dependence have been completely characterized in terms of limit theorems for random sup-measures. For this example, the limit random sup-measure obtained by  \citet{lacaux16time} takes the form 
\equh\label{eq:LS0}
\calM^{\rm sr}_{\alpha,\beta}(\cdot) = \sup_{\ell\in\N}\frac1{\Gamma_\ell^{1/\alpha}}\inddd{\wt R_\ell\topp\beta\cap \,\cdot\, \ne\emptyset},
\eque
where $\{\Gamma_\ell\}_{\ell\in\N}$ are as before, $\{\wt R_\ell\topp\beta\}_{\ell\in\N}$
are i.i.d.~random closed sets of $[0,1]$, each consisting of a randomly shifted $(1-\beta)$-stable  regenerative set (a stable regenerative set is the closure of a stable subordinator; see   Example~\ref{Rem:LS} below  for a complete description of $\wt R_\ell\topp\beta$), and the two sequences are independent. We refer to this family of random sup-measures as {\em stable-regenerative random sup-measures} in this paper. 
More precisely, $\calM^{\rm sr}_{\alpha,\beta}$ arises in limit theorems for a discrete model with parameters $\alpha>0$, $\beta\in(1/2,1)$ \citep{lacaux16time},
and it can be naturally extended to all $\beta\in(0,1)$ (for the original problem in \citep{lacaux16time} with $\beta\in(0,1/2)$, a more complicated random sup-measure of non-Fr\'echet type is shown to arise in the limit in \citep{samorodnitsky17extremal}; note also that a different parameterization $\wt \beta = 1-\beta$ was used in \citep{samorodnitsky17extremal}).

One could draw a comparison between \eqref{eq:completely_random} and \eqref{eq:LS0} by viewing each uniform random variable $U_i$ in \eqref{eq:completely_random} as a random closed set consisting of a singleton point.
From this point of view, for the stable-regenerative random sup-measures, the random closed sets 
$\{\wt R_\ell\topp\beta\}_{\ell\in\N}$ 
represent the limit law of positions of extremes, and in this case they reveal a much more intriguing structure: for example,  each $\wt R_\ell\topp\beta$, as randomly shifted $(1-\beta)$-stable
regenerative set, is uncountably infinite and with Hausdorff dimension $1-\beta$ almost surely. They reflect the picture that each top order statistic shows up at infinitely many different locations, even unbounded if $\calM^{\rm sr}_{\alpha,\beta}$ is viewed as a random sup-measure on $\R$, in a sharp contrast to the situation of independently scattered random sup-measure \eqref{eq:completely_random} where each top order statistic occurs at a unique random location. 

We refer to the phenomena that each top order statistic may show up at multiple and possibly infinitely many  locations by {\em long-range clustering}.  Clustering of extremes have been studied before, but in most examples clusters are {\em local}  in the sense that, roughly speaking, each top order statistic is replaced by a cluster consisting of several correlated values at the same time point, due to certain local dependence structure of the original model (see e.g.~\citep{hsing98exceedance,leadbetter83extremes}).

In this paper, by examining another model of heavy-tailed stationary processes, we prove the convergence of empirical random sup-measures to a family of random sup-measures, exhibiting long-range clustering.  We refer to this family as the {\em Karlin random sup-measures}, denoted by $\mab$ with $\alpha>0$,  $\beta\in(0,1)$ (see \eqref{eq:Karlin_rsm01}).
These random sup-measures are also in the form of \eqref{eq:LS0}: now each random closed set $\wt R_\ell\topp\beta$ is replaced by a new one
consisting of a random number of independent uniform random variables, and hence its complexity is between the independently scattered random sup-measures \eqref{eq:completely_random} and stable-regenerative random sup-measures \eqref{eq:LS0}. In the literature, the Karlin random sup-measures have been considered recently by \citet{molchanov16max}, from the aspect of extremal capacity functionals. 

The Karlin random sup-measures arise in our investigation on the so-called heavy-tailed Karlin model, a variation of an infinite urn scheme investigated by \citet{karlin67central}. The model is a stationary heavy-tailed process where each top order statistic shows up at possibly multiple locations. It has been known to have long-range dependence, and functional central limit theorems 
for related partial sums have been recently investigated in \citep{durieu16infinite,durieu17infinite}. Here, for the extremes, we establish a limit theorem (Theorem~\ref{thm:1}) of point-process convergence encoding the values and corresponding locations of the stationary process  as in \eqref{eq:independent}, with now locations represented by random closed sets.  In particular, the joint convergence describes the long-range clustering of the corresponding order statistics of the Karlin model, and as an immediate consequence the convergence of the empirical random sup-measure to the Karlin random sup-measure in the form of \eqref{eq:LS0} follows (Theorem \ref{thm:2}).

Another way to distinguish the Karlin random sup-measures from independently scattered and stable-regenerative random sup-measures is by noticing that they all have different ergodic properties. This can be understood by comparing the ergodic properties of the induced {\em max-increment processes} of each class. Each max-increment process of a max-stable random sup-measure is a stationary max-stable process. 
Ergodic properties of stationary max-stable processes have been recently investigated in the literature \citep{dombry17ergodic,kabluchko10ergodic,stoev08ergodicity,kabluchko09spectral}. In particular, it is known that the max-increment processes of independently scattered random sup-measures are mixing, those of stable-regenerative random sup-measures are ergodic but not mixing, and here we show that those of Karlin random sup-measures are not ergodic. 

We also notice that the Karlin random sup-measures and stable-regenerative random sup-measures both have the same extremal process as a time-changed standard $\alpha$-Fr\'echet extremal process, and this holds in a much greater generality. 
It is easy to see that the extremal process contains much less information than the corresponding random sup-measure. 
Here we elaborate the relation of the two by showing
that for all self-similar Choquet random sup-measures (defined in Section~\ref{sec:choquet}), 
the associated extremal processes are time-changed standard extremal processes (Proposition \ref{prop:extremal} in the appendix).

The paper is organized as follow. A general class of random sup-measures, the so-called Choquet random sup-measures, is presented in Section~\ref{sec:choquet}. 
In Section \ref{sec:RSM}, we introduce the Karlin random sup-measures. In Section \ref{sec:Karlin}, we introduce the heavy-tailed Karlin model, and state our main results. The proofs are provided in Section \ref{sec:proofs}. In Section \ref{sec:discussions} we discuss related random sup-measures having the same extremal process.
The appendix is devoted to a general result concerning the relation between Choquet random sup-measures and their extremal processes.
Some related background on random closed sets and random sup-measures are provided below.

\subsection*{Preliminary background}
We start with spaces of closed sets. Our main reference is \citep{molchanov05theory}. We shall consider the space of all closed subsets of a given metric space $E$, denoted by $\calF(E)$, with only $E = [0,1]$, $\R$ or $\R_+ :=[0,\infty)$ in this paper. The space $\calF \equiv \calF(E)$ is equipped with the Fell topology. That is, letting $\calG \equiv \calG(E)$ and $\calK \equiv \calK(E)$ denote the open and compact subsets of $\calF$, respectively, the topology generated by the base of sets
\[
\calF_G:=\ccbb{F\in\calF: F\cap G\ne\emptyset},\quad G\in\calG
\]
and
\[
\calF^K:=\ccbb{K\in\calK:F\cap K = \emptyset},\quad K\in\calK.
\]
The Fell topology is also known as the hit-and-miss topology.
With our choice of $E$ (and more generally when it is locally compact  and Hausdorff second countable), the Fell topology on $\calF(E)$ is metrizable. Hence we define random closed sets as random elements in a metric space \citep{billingsley99convergence}. 
The law of a random closed set $R$ is uniquely determined by 
\[
\varphi(K):=\proba(R\cap K\ne\emptyset), \quad  K\in \calK(E),
\]
where $\calK(E)$ is the collection of all compact subsets of $E$, and $\varphi$ is known as the {\em capacity functional} of $R$. 
Let $\indn R$ and $R$ be a collection of random closed sets in $\calF$. A practical sufficient condition for the weak convergence $R_n\weakto R$ in $\calF(E)$ as $n\to\infty$ is that
\[
\limn \proba(R_n\cap A\ne \emptyset) = \proba(R\cap A \ne \emptyset),
\]
for all $A\subset E$ which is a finite union of bounded open intervals such that $\proba(R\cap \wb A\ne\emptyset) = \proba(R\cap A\ne\emptyset)$ where $\wb A$ is the closure set of $A$ \citep[Corollary 1.6.9]{molchanov05theory}.

Next, we review basics on sup-measures on a metric space $E$. Our main references are \citep{obrien90stationary,vervaat97random}. A sup-measure $m$ on $E$ is defined as a set function from $\calG \equiv \calG(E)$ to $\R_+$ (in general the sup-measure could take negative values, but not in the framework of this paper), and it can be uniquely extended to a set function from {\em all} subsets of $E$ to $\R_+$. We start by recalling the definition of a sup-measure on $\calG$. A set function $m:\calG\to\R_+$ is a sup-measure, if  $m(\emptyset) = 0$ and 
\[
m\pp{\bigcup_\alpha G_\alpha} = \sup_\alpha m(G_\alpha)
\]
for all arbitrary collection of $\{G_\alpha\}_\alpha\subset\calG$. Let $\sm(E)$ denote the space of sup-measures from $\calG\to\R_+$. 
The canonical extension of $m:\calG \to\R_+$ to a sup-measure on all subsets of $E$ is given by
\[
m(A) := \inf_{G\in\calG, A\subset G}m(G)
 \quad\text{for all } A\subset E, A\ne\emptyset.
\]

The {\em sup-vague topology} on $\sm(E)$ is defined such that for $\indn m$ and $m$ elements of $\sm(E)$, $m_n\to m$ as $n\to\infty$ if the following two conditions hold
\begin{align*}
\limsupn m_n(K) & \le m(K), \mfa K\in\calK(E),\\
\liminfn m_n(G) & \ge m(G), \mfa G\in\calG(E).
\end{align*}
This choice of topology makes $\sm(E)$ compact and metrizable. 
We then define random sup-measures again as random elements in a metric space. In particular, $M:\Omega\to \sm(E)$ is a random sup-measure, if and only if $M(A)$ is a $\R_+$-valued random variable for all open bounded intervals $A$ or all compact intervals $A$, with rational end points. Examples of particular importance for us include {\em scaled indicator} random sup-measures in the form of
\[
\zeta\inddd{R\cap\,\cdot\,\ne\emptyset},
\]
where $\zeta$ is a positive random variable and $R$ a random closed set, the two not necessarily independent, and the maximum of a finite number of such scaled-indicators.
A practical sufficient condition for weak convergence in $\sm(E)$ is the following. Let $\indn M$ and $M$ be random sup-measures on $E$. We have $M_n\weakto M$ in $\sm(E)$, if 
\[
(M_n(A_1),\dots,M_n(A_d))\weakto (M(A_1),\dots,M(A_d)),
\]
for all bounded open intervals $A_1,\dots,A_d$ of $E$ such that $M(A_i) = M(\wb A_i)$ with probability one \citep[Theorem 3.2]{obrien90stationary}.

Of particular importance among random sup-measures are {\em Fr\'echet (max-stable) random sup-measures}, which are random sup-measures with Fr\'echet finite-dimensional distributions. Recall that a random variable $Z$ has an $\alpha$-Fr\'echet distribution if $\proba(Z\le z) = \exp(-\sigma z^{-\alpha})$, $z>0$, for some constants $\sigma>0$, $\alpha>0$. A random vector $(Z_1,\dots,Z_d)$ has an $\alpha$-Fr\'echet distribution if all its max-linear combinations $\max_{i=1,\dots,d}a_iZ_i$, for $a_1,\dots,a_d>0$, have $\alpha$-Fr\'echet distributions. Now, a random sup-measure is $\alpha$-Fr\'echet if its joint law on finite sets is $\alpha$-Fr\'echet. Equivalently, an $\alpha$-Fr\'echet random sup-measure on $E$ can be viewed as a set-indexed $\alpha$-Fr\'echet max-stable process $\{M(A)\}_{A\subset E}$, that is, a stochastic process of which every finite-dimensional distribution is $\alpha$-Fr\'echet. Fr\'echet random variables and Fr\'echet processes are fundamental objects in extreme value theory, as they arise in limit theorems for extremes of heavy-tailed models  \citep{dehaan84spectral,kabluchko09spectral,resnick87extreme}.

\section{Choquet random sup-measures}\label{sec:choquet} 
A special family of Fr\'echet random sup-measures is the so-called {\em Choquet random sup-measures}, recently introduced by \citet{molchanov16max}. 
It is known that every $\alpha$-Fr\'echet random sup-measure $M$ has the expression
\equh\label{eq:ecf0}
 \proba(M(K)\le z) = \exp\pp{-\frac{\theta(K)}{z^\alpha}},\quad K\in\calK(E),
\eque
where $\theta(K)$ is referred to as the {\em extremal coefficient functional} of $M$. In general, different Fr\'echet random sup-measures may have the same extremal coefficient functional. Given an extremal coefficient functional $\theta$, the so-called Choquet random sup-measure was introduced and investigated in \citep{molchanov16max}, in the form of
\equh\label{eq:Choquet}
\calM(\cdot)\eqd \sup_{\ell\in\N}\frac1{\Gamma_\ell^{1/\alpha}}\inddd{R_\ell\cap \,\cdot\,\ne\emptyset}.
\eque
Here $\{(\Gamma_\ell,R_\ell)\}_{\ell\in\N}$ is a measurable enumeration of points from a Poisson point process on $(0,\infty)\times \calF(E)$ with intensity $dxd\nu$, where $\nu$ is a locally finite measure on $\calF(E)$ uniquely determined by \[
 \nu(\calF_K) \equiv \nu(\ccbb{F\in\calF(E):F\cap K 
 \neq
  \emptyset}) =  \theta(K),\quad K\in\calK(E).
 \] 
The so-defined $\calM$ in \eqref{eq:Choquet} turns out to be an $\alpha$-Fr\'echet random sup-measure with extremal coefficient functional $\theta$, and furthermore its law is uniquely determined by $\theta$.  
It was demonstrated in \citep{molchanov16max} that this family of random sup-measures plays a crucial role among all Fr\'echet random sup-measures from several aspects, and the Choquet theorem plays a fundamental role in this framework, which explains the name. 

In view of limit theorems, Choquet random sup-measures arise naturally in the investigation of extremes of a stationary sequence, including the independently scattered and stable-regenerative random sup-measures (see \eqref{eq:completely_random} and \eqref{eq:LS0} respectively). 
In extreme value theory, many limit theorems are established in terms of extremal processes rather than random sup-measures. Given a general random sup-measure $\calM$, let $\M(t):=\calM([0,t])$, $t\ge0$, denote its associated extremal process. It is well known that $\M$ contains much less information than $\calM$ in general. This is particularly the case in the framework of self-similar Choquet random sup-measures, as their extremal processes are necessarily time-changed versions of a standard $\alpha$-Fr\'echet extremal process.
Recall that a random sup-measure $\calM$ is $H$-self similar for some $H>0$ if 
\equh\label{self-sim}
\calM(\lambda \, \cdot) \eqd \lambda^{H}\calM(\cdot), \mfa \lambda>0.
\eque
By {\em standard $\alpha$-Fr\'echet extremal process}, we mean the extremal process determined by 
the independently scattered random sup-measure $\calM_\alpha$, $\M_\alpha(t):=\calM_\alpha([0,t])$. That is, using the same $\{(\Gamma_\ell,U_\ell)\}_{\ell\in\N}$ as in \eqref{eq:completely_random},
\equh\label{eq:def_ext_pro}
\M_\alpha(t):=
\sup_{\ell\in\N}\frac1{\Gamma_\ell^{1/\alpha}}\inddd{U_\ell\le t},\quad  t\ge 0.
\eque
\begin{Prop}\label{prop:ssChoquet}
For any $H$-self-similar Choquet $\alpha$-Fr\'echet random sup-measure $\calM$ with $H>0$, the corresponding extremal process $\M$ satisfies
\[
\theta([0,1])\ccbb{\M(t)}_{t\ge 0}\eqd  \ccbb{\M_{\alpha}(t^{\alpha H})}_{t\ge 0}.
\]
\end{Prop}
To the best of our knowledge, this fact has not been noticed in the literature before. 
This proposition actually follows from a more general result on Choquet random sup-measures and the corresponding extremal processes, which is of independent interest and established in Proposition \ref{prop:extremal} in the appendix.
In the upcoming setting, this provides another justification that it is important to work with random sup-measures in the presence of long-range dependence, as the corresponding extremal processes capture much less information of the dependence. See also the discussion in Section~\ref{sec:discussions}.

\section{Karlin random sup-measures}\label{sec:RSM}
In this section we provide two representations of Karlin random sup-measures. 
They are Choquet random sup-measures with $\alpha$-Fr\'echet marginals and they depend on a second parameter $\beta\in(0,1)$. 

Let us denote by $xA$, for $x>0$ and $A\subset\R$, the scaled set $\{xy: y\in A\}$. The Karlin random sup-measure $\calM_{\alpha,\beta}$ on $\R$ is defined by the following representation
\equh\label{eq:Karlin_rsm}
\calM_{\alpha,\beta}(A) := \sup_{\ell\in\N}\frac1{\Gamma_\ell^{1/\alpha}}\inddd{\wt\calN_\ell(x_\ell A)\ne0},\quad A\in\calG(\R),
\eque
where $\{(\Gamma_\ell,x_\ell,\wt\calN_\ell)\}_{\ell\in\N}$ is an enumeration of the points from a Poisson point process on $\R_+\times\R_+\times \mathfrak M_+(\R)$ with intensity measure $d\gamma\times \Gamma(1-\beta)^{-1}\beta x^{-\beta-1}dx \times d\wt\proba$. Here $\mathfrak M_+(\R)$ is the space of Radon point measures on $\R$ and $\wt\proba$ is the probability measure on it induced by a standard Poisson random measure (with intensity $dx$). 
Equivalently, the Poisson point process $\{(\Gamma_\ell,x_\ell,\wt\calN_\ell)\}_{\ell\in\N}$ can be viewed as the Poisson point process $\{(\Gamma_\ell,x_\ell)\}_{\ell\in\N}$  on $\R_+\times\R_+$ with intensity $d\gamma \times \Gamma(1-\beta)^{-1}\beta x^{-\beta-1}dx$ and i.i.d.~marks $\{\wt \calN_\ell\}_{\ell\in\N}$ with law $\wt\proba$.

To see that $\mab$ is a Choquet random sup-measure, we introduce the random closed set $\wt \calR_\ell$ induced by $\wt\calN_\ell$ as
\[
\wt\calR_\ell:=\ccbb{t\in\R:\wt\calN_\ell(\{t\}) = 1},
\]
and then write
$\sccbb{\wt \calN_\ell(x_\ell A)\ne 0} = \sccbb{(\wt \calR_\ell/x_\ell)\cap A\ne\emptyset}$.
So \eqref{eq:Karlin_rsm} now becomes
\[
\mab(A) = \sup_{\ell\in\N}\frac1{\Gamma_\ell^{1/\alpha}}\inddd{\pp{\wt\calR_\ell/x_\ell}\cap A\ne\emptyset}, \quad A\in\calG(\R),
\]
as in \eqref{eq:Choquet}, and then it can be extended to all $A\subset\R$ by the canonical extension of sup-measures. 

Viewing $\{\calM_{\alpha,\beta}(A)\}_{A\subset\R}$ as a set-indexed $\alpha$-Fr\'echet max-stable process, we have the following joint distribution:
\begin{multline}\label{eq:fdd}
\proba\pp{\mab(A_1)\le z_1,\dots,\mab(A_d)\le z_d} \\
= \exp\pp{-\Gamma(1-\beta)^{-1}\int_0^\infty \beta x^{-\beta-1}\wt\esp \pp{\bigvee_{i=1}^d\frac{\inddd{\wt\calN(xA_i)\ne0}}{z_i^\alpha }}dx},
\end{multline}
for all $d\in\N$, $z_1,\dots,z_d>0$, where  $\wt\esp$ is the expectation with respect to $\wt \proba$.
   See \citep{stoev06extremal,molchanov16max} for more details. 
It suffices to consider $A_1,\dots,A_d$ as open (or compact) intervals in $\R$ (not necessarily disjoint) above to determine the law of $\mab$. 

From the above presentation, in particular we compute for $d=1$ and a compact set $K\subset\R$, 
\[
\proba(\mab(K)\le z)  = \exp\pp{-\Gamma(1-\beta)^{-1}\int_0^\infty \beta x^{-\beta-1}\wt\proba\pp{\wt \calN(xK)\ne 0} dxz^{-\alpha}}.
\]
Let $\Leb$ denote the Lebesgue measure on $\R$. We have
\begin{align}
 \int_0^\infty \beta x^{-\beta-1}\wt\proba\pp{\wt \calN(xK)\ne 0} dx\label{eq:ecf}
& = \int_0^\infty \beta x^{-\beta-1}\pp{1-\exp(-x \Leb(K))} dx\\
& = \Leb(K)\int_0^\infty x^{-\beta}\exp(-x\Leb(K))dx\nonumber\\
& = \Gamma(1-\beta)\Leb(K)^\beta. \nonumber
\end{align}
Therefore we arrive at, for all $z>0$,
\[
\proba(\mab(K)\le z) = \exp\pp{-\frac{\theta_\beta(K)}{z^\alpha}} \qmwith \theta_\beta(K):= {\rm Leb}(K)^\beta.
\]
The function $\theta_\beta$ is the extremal coefficient functional of the random sup-measure $\mab$.

It is clear from the definition \eqref{eq:fdd} that $\mab$ is $\beta/\alpha$-self-similar in the sense 
of \eqref{self-sim}
and translation-invariant
\[
\mab(\cdot) \eqd \mab(x+\cdot), \mfa x\in\R.
\]
It is also remarkable that it is symmetric (or rearrangement invariant \citep[Sect.~9]{molchanov16max}) 
in the sense that its law only depends on the Lebesgue measures of the sets.
More precisely, for two collections of disjoint open intervals $\{A_1,\dots,A_d\}$ and $\{B_1,\dots,B_d\}$ such that $\Leb(A_i) = \Leb(B_i), i=1,\dots,d$, we have
\[
\pp{\mab(A_1),\dots,\mab(A_d)}\eqd \pp{\mab(B_1),\dots,\mab (B_d)}.
\]
This is a stronger notion than the translation invariance, which has been known to hold true for all random sup-measures arising from stationary sequences \citep{obrien90stationary}.

By self-similarity essentially all properties of $\mab$ can be investigated by restricting to a bounded interval, in which case $\mab$ has a more convenient representation. We consider its restriction to $[0,1]$ here. In this case, $\theta_\beta$ determines the law of a random closed set $\calR\topp\beta$ in $[0,1]$ by
\equh\label{eq:R}
\proba(\calR\topp\beta\cap K\ne\emptyset) = \frac{\theta_\beta(K)}{\theta_\beta([0,1])} = \Leb(K)^\beta, \mfa K\subset[0,1] \mbox{ compact}.
\eque

Now, restricting to $[0,1]$, it follows that 
\equh\label{eq:Karlin_rsm01} 
\calM_{\alpha,\beta} (\cdot)
 \eqd  \sup_{\ell\in\N}\frac1{\Gamma_\ell^{1/\alpha}}\inddd{\pp{\calR\topp\beta_\ell\cap\,\cdot\,}\ne\emptyset}\,\text{ on }[0,1],
\eque
where $\{\Gamma_\ell\}_{\ell\in\N}$ is the sequence of arrival times of a standard Poisson point process on $\R_+$, $\{\calR\topp\beta_\ell\}_{\ell\in\N}$ are i.i.d.~copies of $\calR\topp\beta$,  and the two sequences are independent. The fact that $\mab$ in \eqref{eq:Karlin_rsm} has the same presentation (in law) as in \eqref{eq:Karlin_rsm01} when restricted to $[0,1]$, follows from either a straightforward computation of finite-dimensional distributions of random sup-measures based on \eqref{eq:Karlin_rsm01}, or from a more general property of Choquet random sup-measures \citep[Corollary 4.5]{molchanov16max}.

In addition, we have the following probabilistic representation of $\calR\topp\beta$. 
\begin{Lem}\label{lem:R}
Suppose $\beta\in(0,1)$. Let $Q_\beta$ be an $\N$-valued random variable with probability mass function
\[
\proba(Q_\beta = k) = \frac{\beta(1-\beta)_{(k-1)\uparrow}}{k!}=:p_\beta(k),\quad k\in\N,
\]
with $(a)_{n\uparrow} = a(a+1)\cdots(a+n-1)$, $n\in\N$, $a\in\R$. 
Let $\indn U$ be i.i.d.~random variables uniformly distributed over $(0,1)$, independent from $Q_\beta$. Then,
\[
\calR\topp\beta \eqd \bigcup_{i=1}^{Q_\beta} \{U_i\}. 
\]
\end{Lem}
\begin{proof}
It suffices to prove that $\bigcup_{i=1}^{Q_\beta}\{U_i\}$ has the same capacity functional as $\calR\topp\beta$ in \eqref{eq:R}. We have, by first conditioning on $Q_\beta$,
\[
\proba\pp{\pp{\bigcup_{i=1}^{Q_\beta}\{U_i\}}\cap K\ne\emptyset} = \esp\bb{1-(1-\Leb(K))^{Q_\beta}}.
\]
One can show that the prescribed distribution of $Q_\beta$ satisfies the property, for all $z\in(0,1)$, 
\[
1-z^\beta = \esp\bb{(1-z)^{Q_\beta}}.
\]
See for example \citep[Eq.~(3.42)]{pitman06combinatorial}. In view of \eqref{eq:R}, this completes the proof.
\end{proof}

\begin{Rem}
The law of $Q_\beta$ has been known to be related to the Karlin model defined in Section~\ref{sec:Karlin}, and hence it is not a coincidence that it shows up in the limit random sup-measure. In fact, $Q_\beta$ is a {\em size-biased sampling} from the asymptotic frequency $\{p_\beta(k)\}_{k\in\N}$ 
of blocks of size $k$ of an infinite exchangeable random partition with $\beta$-diversity.
See \citep[Section 3.3]{pitman06combinatorial} for more details and Remark \ref{rem:pa} below. 
\end{Rem}

\begin{Rem}
The first representation of $\calM_{\alpha,\beta}$ has been already considered by \citet{molchanov16max}. Their description starts with and focuses on the extremal coefficient functional $\theta_\beta$ whereas we start from the underlying Poisson point process directly. This is suggested in \citep[Remark 9.8]{molchanov16max}, while more detailed discussions can be found in the first arXiv online version of the same paper. In particular, Example 9.5 therein provides the same representation as in \eqref{eq:Karlin_rsm}. The interpretation of the set $\calR\topp\beta$ in our Lemma~\ref{lem:R} seems to be new. 
\end{Rem}

The Karlin random sup-measures also interpolate between the independently scattered random sup-measures $\calM_\alpha$
and the completely dependent one, defined as $\calM_\alpha^{\rm c}(\cdot) = Z\inddd{\,\cdot\,\ne\emptyset}$ for a standard $\alpha$-Fr\'echet random variable $Z$ (the random sup-measure taking the same value $Z$ on any non-empty set).

\begin{Prop}For every $\alpha>0$,  $\calM_{\alpha,\beta} \weakto\calM_\alpha$ as $\beta\uparrow 1$, and  $\calM_{\alpha,\beta} \weakto \calM_\alpha^{\rm c}$ as $\beta\downarrow 0$.
\end{Prop}
\begin{proof}
It suffices to notice that by the capacity functional in \eqref{eq:R},
$\calR\topp\beta\weakto U$ as $\beta\uparrow1$ where $U$ is the random closed set induced by the uniform random variable on $(0,1)$, and $\calR\topp\beta\weakto[0,1]$, a deterministic set, as $\beta\downarrow 0$. 
\end{proof}

We conclude this section by examining the ergodic properties of $\mab$.
Every self-similar and translation invariant random sup-measure $\calM$ naturally induces a stationary process, the so-called {\em max-increment process} defined as
\equh\label{eq:max-increment}
\zeta(t):=\calM((t-1,t]), \quad t\in\R.
\eque
\begin{Prop}
The max-increment process $\{\zeta_{\alpha,\beta}(t)\}_{t\in\R}$  of $\mab$ is not ergodic.
\end{Prop}
\begin{proof}
Introduce, for $z>0$, $t\in\R$,
\[
\tau_z(t):=\log\proba(\zeta_{\alpha,\beta}(0)\le z, \zeta_{\alpha,\beta}(t)\le z) - 2\log\proba(\zeta_{\alpha,\beta}(0)\le z).
\]
A simple necessary and sufficient condition for ergodicity of a stationary $\alpha$-Fr\'echet process is that 
\[
\lim_{T\to\infty}\frac1T\int_0^T\tau_z(t)dt = 0 \mfa z>0,
\]
see \citet{kabluchko10ergodic}.
Here we have, for $t>1$, 
\begin{align*}
-\log &\proba(\zeta_{\alpha,\beta}(0)\le z,\zeta_{\alpha,\beta}(t)\le z)\\
 & = \frac1{z^\alpha}\Gamma(1-\beta)^{-1}\int_0^\infty \beta x^{-\beta-1}\wt \proba\pp{\wt \calN(x(-1,0])\ne 0, \wt\calN(x(t-1,t])\ne 0}dx
\\
& = \frac1{z^\alpha}\Gamma(1-\beta)^{-1}\int_0^\infty \beta x^{-\beta-1}(1-e^{-x})^2dx  = (2-2^\beta)z^{-\alpha}.
\end{align*}
In addition to \eqref{eq:ecf}, this implies for all $t>1$, $z>0$,
\begin{align*}
\tau_z(t) & = \bb{2\theta_\beta((-1,0]) - (2-2^\beta)}z^{-\alpha}
  = 2^\beta z^{-\alpha}.
\end{align*}
The desired result hence follows.
\end{proof}

\section{A heavy-tailed Karlin model}\label{sec:Karlin}
In this section, we introduce a discrete stationary process $\indn X$ based on a model, originally studied by \citet{karlin67central}, which is essentially an infinite urn scheme.  Here, we shall work with a heavy-tailed randomized version of the original model.

To start with, consider an $\N$-valued random variable $Y$ with $\proba(Y = k) = p_k$, $k\in\N$. We assume that $p_1\ge p_2\ge\cdots >0$ and, for technical purpose, encode them into the measure
\begin{equation}\label{nu}
\nu := \sif\ell1\delta_{1/p_\ell},
\end{equation}
where $\delta_x$ is the unit point mass at $x$. The following regular variation assumption is made on the frequencies: 
\equh\label{eq:RVp}
\nu((0,x]) = \max\{\ell\in\N: 1/p_\ell\le x\} = x^\beta L(x) \quad \text{ with } \beta\in(0,1),
\eque
for some slowly varying function $L$ at infinity. 

The randomized Karlin model $\indn X$ is defined through a two-layer construction. We imagine that there are infinitely many empty boxes indexed by $\N$.
First, we independently associate a heavy-tailed random variable to each box. Second, 
at each round $n$, we throw a ball at random in one of the boxes (according to 
the law of $Y$) and we consider the corresponding heavy-tailed random variable as the value of our process at time $n$. Namely,
let $\{\varepsilon_\ell\}_{\ell\in\N}$ be i.i.d.~random variables with common law such that
\equh\label{eq:RVepsilon}
\proba(\varepsilon_1>y) \sim c_\alpha y^{-\alpha} \mmas y\to\infty \quad \text{ with }  \alpha>0,\, c_\alpha\in(0,\infty),
\eque
each associated with the box with label $\ell\in\N$.
Let $\indn Y$ be i.i.d.~random variables with common law as $Y$ described above, independent of $\{\varepsilon_\ell\}_{\ell\in\N}$.
The stationary sequence $\indn X$ is then obtained by setting
\[
X_n:= \varepsilon_{Y_n}, n\in\N.
\]

Here, we are interested in the empirical random sup-measure of $\indn X$ on $[0,1]$ introduced as
\[
M_n(\cdot) := \max_{i/n\in\,\cdot\,} X_i,
\]
and its limit as $n\to\infty$. 
Important quantities relying on the infinite urn scheme are, 
\[
K_{n,\ell}:=\summ i1n \inddd{Y_i = \ell},\;\ell\ge 1, \qmand K_n := \sif \ell1\inddd{K_{n,\ell}\ne\emptyset},
\]
the number of balls in the box $\ell$ and the number of non-empty boxes at time $n$, respectively.
We know from \citep{karlin67central} that, under \eqref{eq:RVp}, $K_n\sim \Gamma(1-\beta) n^\beta L(n)$ almost surely.

For a more detailed description of the model, we shall work within the framework of point-process convergence generalizing \eqref{eq:independent}.  For each $n\in\N$, introduce, for $\ell\ge 1$,
\[
R_{n,\ell}=\{i\in\{1,\ldots,n\} : Y_i=\ell\}.
\]
The following point process $\xi_n$ on $\R_+\times \calF([0,1])$ encode the information of our random model at time $n$:
\begin{equation}\label{eq:xi_n}
 \xi_n
 :=\sum_{\ell\ge 1,\, K_{n,\ell}\ne0 }\delta_{\pp{\varepsilon_{\ell}/b_n,R_{n,\ell}/n}},
\end{equation}
The first coordinate in the Dirac masses presents the value (normalized by $b_n$, given below) attached to the box $\ell$ and the second coordinate the possible multiple positions among $\{1,\dots,n\}$ (standardized by $1/n$) where this box has been chosen.

Our main results are the following. The first is a complete point-process convergence.
\begin{Thm}
\label{thm:1}
For the model above under assumptions \eqref{eq:RVp} and \eqref{eq:RVepsilon}, with
\equh\label{eq:bn}
b_n := (c_\alpha\Gamma(1-\beta)n^\beta L(n))^{1/\alpha},
\eque
we have
\[
\xi_n\weakto \xi:=\sif\ell1 \delta_{\pp{\Gamma_\ell^{-1/\alpha},\calR_\ell\topp\beta}},\mmas n\to\infty,
\]
in $\mathfrak M_+((0,\infty)\times\calF([0,1]))$, 
where $\{(\Gamma_\ell,\calR_\ell\topp\beta)\}_{\ell\in\N}$ have the same law as in \eqref{eq:Karlin_rsm01}. 
\end{Thm}
The second is the convergence of random sup-measures.
\begin{Thm}\label{thm:2}Under the assumption of Theorem \ref{thm:1}, we have
\[
\frac{1}{b_n}M_n\weakto \mab,\mmas n\to\infty,
\]
in $\sm([0,1])$. 
\end{Thm}
Theorem \ref{thm:1} is proved by analyzing the top order statistics of the model and their locations. Theorem \ref{thm:2} is a direct corollary of Theorem \ref{thm:1}. Nevertheless, we will also give a second proof of it which is straightforward, without any analysis of the other top order statistics except the largest.

\begin{Rem}\label{rem:pa}
In the representation of the law of $\calR\topp\beta$ in Lemma \ref{lem:R}, the probability mass function $\{p_\beta(k)\}_{k\in\N}$ has an intrinsic connection to the Karlin model: each $p_\beta(k)$ is the asymptotic frequency of the number of boxes with exactly $k$ balls, namely 
\[
\limn\frac1{K_n}\sif\ell1\inddd{K_{n,\ell} = k} = p_\beta(k) \;\mbox{ a.s.}
\]
This has been known since \citet{karlin67central}. 
\end{Rem}
\begin{Rem}
For the sake of simplicity, we do not introduce a slowly varying function in \eqref{eq:RVepsilon} as in the common setup for heavy-tailed random variables.
Replacing \eqref{eq:RVepsilon} by
\[
\proba(\varepsilon_1>y) \sim y^{-\alpha}\ell(y) \mmas y\to\infty
\]
with $\alpha>0$ and $\ell$ a slowly varying function, 
the same limit arises while the correct normalization would involve the Bruijn conjugate (e.g.~\citep[Proposition 1.5.15]{bingham87regular}).
\end{Rem}

\section{Proofs}\label{sec:proofs}
In order to analyze the point process $\xi_n$, we introduce a description of it through the extreme values of the Karlin model.
For each $n\in\N$, we consider the $K_n$ random variables
\[
\ccbb{\varepsilon_\ell:K_{n,\ell}\ne 0}
\]
and their order statistics denoted by 
\[
\varepsilon_{n,1}\ge \cdots\ge \varepsilon_{n,K_n}.
\]
When there are no ties, we let $\what\ell_{n,k}$ denote the label of the box corresponding to the $k$-th order statistics, so that
\[
\varepsilon_{n,k} = \varepsilon_{\what \ell_{n,k}}, \mbox{ for } k\le K_n,
\]
and set $\what\ell_{n,k}:=0$ for $k>K_n$.
When there are ties among the order statistics, the aforementioned labeling is no longer unique, and we choose one 
at random among all possible ones in a uniform way.
This procedure guarantees the independence between the values of the order statistics and the permutation that classifies them. That is, given $K_n$, the variables $\what\ell_{n,1},\ldots,\what\ell_{n,K_n}$ are independent of the variables $\varepsilon_{n,1},\ldots,\varepsilon_{n,K_n}$.
Now, introduce the random closed sets
\[
\what R_{n,k}:=\ccbb{i=1,\dots n:Y_i = \what \ell_{n,k}},\quad k=1,\dots,K_n,
\]
and $\what R_{n,k} :=\emptyset$ if $k>K_n$. 
The point processes $\xi_n$ introduced in \eqref{eq:xi_n} can then be written as
\[
\xi_n 
= \summ k1{K_n} \delta_{\pp{\varepsilon_{n,k}/b_n,\what R_{n,k}/n}}.
\]
The key step in our proof is to investigate the following approximations of $\xi_n$, keeping only the top order statistics,
\[
\xi_n\topp m:= \summ k1m \delta_{\pp{\varepsilon_{n,k}/b_n, \what R_{n,k}/n}},\quad m\in\N.
\]
Here and below, we set $\varepsilon_{n,k} := 0$ if $k>K_n$. 
\begin{Prop}\label{prop:m}
For all $m\in\N$, we have
\[
\xi_n\topp m\weakto \xi\topp m :=\summ \ell 1m \delta_{\pp{\Gamma_\ell^{-1/\alpha},\calR_\ell\topp\beta}},\mmas n\to\infty,
\]
in $\mathfrak M_+((0,\infty)\times\calF([0,1]))$, 
where $\{(\Gamma_\ell,\calR_\ell\topp\beta)\}_{\ell\in\N}$ have the same law as in \eqref{eq:Karlin_rsm01}.
\end{Prop}

\begin{proof}There is only a finite number of random points in both $\xi_n\topp m$ and $\xi\topp m$. Hence, it suffices to prove the joint convergence
\equh\label{eq:fdd2}
\pp{\frac{\varepsilon_{n,1}}{b_n},\dots,\frac{\varepsilon_{n,m}}{b_n},\frac{\what R_{n,1}}n,\dots,\frac{\what R_{n,m}}n}\weakto \pp{\Gamma_1^{-1/\alpha},\dots,\Gamma_m^{-1/\alpha},\calR_1\topp\beta,\dots,\calR_m\topp\beta} 
\eque
in $\R_+^m\times \calF([0,1])^m$, as $n\to\infty$. 
Under the heavy-tail assumption \eqref{eq:RVepsilon}, the convergence of the first $m$ coordinates, as the normalized $m$ top order statistics of $K_n$ i.i.d.\ random variables, is well known from \citep{lepage81convergence} if $K_n$ is a deterministic sequence increasing to infinity and the normalization (here $b_n$) is $c_\alpha^{1/\alpha} K_n^{1/\alpha}$. For the Karlin model, under the regular variation assumption \eqref{eq:RVp}, it has been shown that 
\[
\limn \frac{K_n}{n^\beta L(n)} = \Gamma(1-\beta)\; \mbox{ a.s.},
\]
see \citep[Corollary 21]{gnedin07notes}.
Therefore the convergence of the first $m$ coordinates follows.
Further, 
on the left-hand side of \eqref{eq:fdd2}, the first and last $m$ coordinates are 
conditionally independent given the event $\{K_n\ge m\}$. Since $\proba\pp{K_n\ge m}\to 1$ as $n\to\infty$, it is sufficient to
prove the convergence of the last $m$ coordinates to conclude.
The main difficulty in the analysis of the last $m$ coordinates is due to 
their dependence. To overcome this difficulty, we first consider a coupled Poissonized version of the model.
Namely, let $\{N(t)\}_{t\ge 0}$ denote a standard Poisson process on $\R_+$ independent of $\indn Y$ and $\indn\varepsilon$, and let $0<\tau_1<\tau_2<\cdots$ denote its consecutive arrival times. 
We consider the coupled model where we shift the fixed locations $1,2,\dots,n$ of the original model to 
the random points corresponding to the consecutive random arrival times of $N$. The Poissonized process is then $\{X_{N(t)}\}_{t\ge 0}$. 
In this way, we set
\begin{equation}\label{eq:def-tildeKn}
\wt K_{n,\ell} := \sif i1\inddd{Y_i = \ell,\,\tau_i \le n} \qmand \wt K_n:= \sif \ell1\inddd{\wt K_{n,\ell}\ne 0}.
\end{equation}
It is important to keep in mind that, for this model, there are $\wt K_n$ different $\varepsilon$ involved at time $n$, instead of $K_n$. Note that,  thanks to the coupling, $\wt K_n=K_{N(n)}$.
Thus, the order statistics of the set $\{\varepsilon_\ell : \wt K_{n,\ell}\ne 0\}$
are exactly $\varepsilon_{N(n),1}\ge \cdots\ge \varepsilon_{N(n),\wt K_{n}}$. 
Now, introduce $\wt \ell_{n,k}$ such that 
\[
\varepsilon_{\wt\ell_{n,k}} = \varepsilon_{N(n),k},\quad k=1,\dots,\wt K_n,
\]
and $\wt\ell_{n,k} := 0$ if $k>\wt K_n$. Again, in case of ties, we choose uniformly a random labeling as before.
Then we define
\[
\wt R_{n,k}:= \ccbb{\tau_i: Y_i = \wt\ell_{n,k}}\cap[0,n],\quad k=1,\ldots, \wt K_n.
\]
The key observation on the Poissonization procedure is that given that  $\wt\ell_{n,1} = \ell_1,\dots,\wt \ell_{n,m}=\ell_m$, with $\ell_1,\dots,\ell_m>0$, $\wt R_{n,1},\dots,\wt R_{n,m}$ are independent random closed sets; this is a consequence of the thinning property of Poisson processes. 
Moreover, the law of each $\wt R_{n,k}$ is the 
conditional law of the set of the arrival times of a Poisson process with intensity $p_{\ell_k}$ within $[0,n]$, given that it is not empty.

We first show that
\equh\label{eq:wtR}
\pp{\frac{\wt R_{n,1}}n,\dots,\frac{\wt R_{n,m}}n} \weakto \pp{\calR_1\topp \beta,\dots,\calR_m\topp\beta}.
\eque
Let $A_1,\dots,A_m$ be $m$ open intervals within $(0,1)$.
We first compute
\begin{multline}\label{eq:Ak}
\proba\pp{\bigcap_{k=1}^m\ccbb{\frac1n\wt R_{n,k}\cap A_k \ne\emptyset}}\\
= \sum_{\ell_1,\dots,\ell_m\in\N}\proba\pp{\bigcap_{k=1}^m\ccbb{\frac1n\wt R_{n,k}\cap A_k \ne\emptyset}\cap\ccbb{ \wt\ell_{n,k} = \ell_k}}. 
\end{multline}
For every  choice of $\ell_1,\dots,\ell_m\in\N$ that are mutually distinct (otherwise the probability above is zero), let $N_k$ be a Poisson process with parameter $p_{\ell_k}$, $k=1,\dots,m$, and $\wt R_k$ the corresponding random closed set induced by its arrival times in $[0,n]$.
Given $\{\wt K_{n,\ell}\}_{\ell\in\N}$, the probability of the event 
$\{\wt\ell_{n,1} = \ell_1,\ldots, \wt\ell_{n,m} = \ell_m\}$ is 
\[
 \indd{\wt K_{n,\ell_1}\ne0,\ldots, \wt K_{n,\ell_m}\ne0 }\frac{(\wt K_n-m)!}{\wt K_n !}
\]
as each non-empty box has equal probability to be the $k$-th largest (above $j!$ stands for the factorial of the non-negative integer $j$).
Therefore
by conditioning on the values of $\{\wt\ell_{n,k}\}_{k=1,\dots,m}$ first, 
and then using the independence of the $\wt K_{n,\ell}$, we have, 
letting $\lambda_k$ denote the Lebesgue measure of $A_k$, 
\begin{align*}
\proba & \pp{\bigcap_{k=1}^m\ccbb{\frac1n\wt R_{n,k}\cap A_k \ne\emptyset}\cap\ccbb{ \wt\ell_{n,k} = \ell_k}}\\
& = \esp\bb{ \frac{(\wt K_n-m)!}{\wt K_n !}\prodd k1m\indd{\wt K_{n,\ell_k}\ne 0}\,\proba\pp{\wt R_k\cap nA_k\ne\emptyset\mid \wt R_k\cap [0,n]\ne\emptyset}}\\
& = \esp\bb{  \frac{(\wt K_n-m)!}{\wt K_n !}\prodd k1m \indd{\wt K_{n,\ell_k}\ne 0}\pp{\frac{1-e^{-\lambda_knp_{\ell_k}}}{1-e^{-np_{\ell_k}}}}}\\
& = \esp \bb{\frac{\wt K_n^{(\ell_1,\ldots,\ell_m)}!}{\spp{m+\wt K_n^{(\ell_1,\ldots,\ell_m)}}!}}\prodd k1m (1-e^{-\lambda_knp_{\ell_k}}),
\end{align*}
where 
\[
 \wt K_n^{(\ell_1,\ldots,\ell_m)}=\sum_{\ell\ge1,\, \ell \not\in\{\ell_1,\ldots,\ell_m\}} \indd{\wt K_{n,\ell}\ne0}.
\]
We shall prove, in Lemma \ref{lem:tildeKn} below, 
that $\wt\Phi_n/((\wt K_n-m)\vee 1)\to 1$ and $\wt\Phi_n/(\wt K_n+m)\to 1$ in $L^m$, where $\wt\Phi_n:=\esp \wt K_n \sim \Gamma(1-\beta) n^\beta L(n)$ according to \citep[Proposition~17 and Lemma~1]{gnedin07notes}. 
Using that 
\[
\frac{1}{(m + \wt K_n)^m}\le \frac{\wt K_n^{(\ell_1,\ldots,\ell_m)}!}{\spp{m+\wt K_n^{(\ell_1,\ldots,\ell_m)}}!}\le\frac{1}{((\wt K_n-m)\vee 1)^m}  ,
\]
we infer that 
\[
 \frac{\wt K_n^{(\ell_1,\ldots,\ell_m)}!}{\spp{m+\wt K_n^{(\ell_1,\ldots,\ell_m)}}!}\sim \frac{1}{\wt \Phi_n^m}\text{ in $L^1$, uniformly in }(\ell_1,\ldots,\ell_m), \mbox{ as } n\to\infty.
\]
The right-hand side of \eqref{eq:Ak} then becomes 
\begin{align}
\sum_{\ell_1,\dots,\ell_m\in\N,\ne} &\esp \bb{\frac{\wt K_n^{(\ell_1,\ldots,\ell_m)}!}{\spp{m+\wt K_n^{(\ell_1,\ldots,\ell_m)}}!}}\prodd k1m (1-e^{-\lambda_knp_{\ell_k}}) \nonumber\\
& \sim \frac1{(\Gamma(1-\beta) n^\beta L(n))^m}\sum_{\ell_1,\dots,\ell_m\in\N,\ne} \prodd k1m\pp{1-e^{-\lambda_knp_{\ell_k}}} \label{eq:ell_ne},\text{ as }n\to\infty,
\end{align}
where in the summation, $\ne$ indicates that $\ell_1,\dots,\ell_m$ are mutually distinct.
If we sum over all $\ell_1,\dots,\ell_m\in\N$ instead, recalling the definition of $\nu$ in \eqref{nu},
we have
\equh\label{eq:ell_all}
\sum_{\ell_1,\dots,\ell_m\in\N}\prodd k1m \pp{1-e^{-\lambda_knp_{\ell_k}}} = \prodd k1m \int_0^\infty (1-e^{-\lambda_k n/x})\nu(dx).
\eque
For the Karlin model, it is well known that the regular variation assumption \eqref{eq:RVp} on $\nu$ leads to, after integration by parts and change of variables,
\begin{align*}
\int_0^\infty(1-e^{-\lambda n/x})\nu(dx) & = \int_0^\infty \frac{\lambda n}{x^2}e^{-\lambda n/x}\nu((0,x])dx\\
& \sim \nu((0,n])\lambda \int_0^\infty x^{\beta-2}e^{-\lambda /x}dx = \nu((0,n])\lambda^\beta\Gamma(1-\beta).
\end{align*}
This gives the asymptotic of \eqref{eq:ell_all}, and also tells that the summations in \eqref{eq:ell_all} and \eqref{eq:ell_ne} are asymptotically equivalent. Therefore, we have shown that
\[
\limn \proba\pp{\bigcap_{k=1}^m\ccbb{\frac1n\wt R_{n,k}\cap A_k \ne\emptyset}} = \prodd k1m \lambda_k^\beta. 
\]
This established the claimed weak convergence in \eqref{eq:wtR}.

To complete the proof, it remains to show that $\wt R_{n,k}/n$ and $\what R_{n,k}/n$ can be made close with arbitrarily high probability by taking $n$ large enough. To make this idea precise, we consider the Hausdorff metric $\dH$ for non-empty compact sets defined as, for two non-empty compact sets $F_1$ and $F_2$,
\[
\dH(F_1,F_2) := \max\ccbb{\sup_{x\in F_1}d(x,F_2),\sup_{x\in F_2}d(x,F_1)},
\]
where $d$ above is the distance between a point and a set induced in $\R$ by Euclidean metric: $d(x,A) := \inf_{y\in A}|x-y|$. It is known that $\dH$ metricizes the Fell topology on $\calF'([0,1]) := \calF([0,1])\setminus\{\emptyset\}$. See for example \citep[Appendix C]{molchanov05theory}. For $n$ large enough, consider the event
\[
B\topp m_n :=\{K_n\ge m\}\cap\{\wt K_n\ge m\},
\]
so that, under $B\topp m_n$, $\what R_{n,k} \ne \emptyset$ and $\wt R_{n,k} \ne \emptyset$ for all $k=1,\ldots,m$.
It is clear that $\limn\proba(B_n\topp m) = 1$. Therefore, \eqref{eq:fdd2} and hence the proposition shall follow from \eqref{eq:wtR} and the fact that for all $\delta>0$, 
\equh\label{eq:dH_approximation}
\limn \proba\pp{\ccbb{\max_{k=1,\dots,m}\dH\pp{\frac{\what R_{n,k}}n,\frac{\wt R_{n,k}}n}>\delta }\cap B_n\topp m} = 0.
\eque
To prove \eqref{eq:dH_approximation}, we first introduce the event 
\[
E_n\topp m =\left\{\what\ell_{n,1}=\wt\ell_{n,1},\ldots,\what\ell_{n,m}=\wt\ell_{n,m} \right\},
\]
and we shall prove that $\limn\proba(E_n\topp m)=1$.
Since the probability of the event
\[
 T_n\topp m:=\ccbb{\text{no ties in the $m+1$ top order statistics of }\{\varepsilon_\ell : K_{n,\ell}\ne 0 \text{ or } \wt K_{n,\ell}\ne 0\}}
\]
goes to $1$ as $n\to\infty$, this will follow if one can show that
\begin{equation}\label{eq:E_n^m}
\lim_{n\to\infty}\proba\pp{E_n\topp m\cap B_n\topp m\cap T_n\topp m}= 1.
\end{equation} 
Assuming $B\topp m_n$ and  $T_n\topp m$, the event $E_n\topp m$ holds if the $m$ top order statistics from the set $\{\varepsilon_\ell: K_{n,\ell}\ne 0 \text{ or } \wt K_{n,\ell}\ne 0\}$ already appear in the subset $\{\varepsilon_\ell: K_{n,\ell}\ne 0 \text{ and } \wt K_{n,\ell}\ne 0\}$.
Given $K_n^\wedge:=K_n\wedge\wt K_n$ and $K^\vee_n:=K_n\vee\wt K_n$, using the fact that the locations (labellings) of the order statistics among $\{1,\dots,K^\vee_n\}$ are uniformly distributed,
the desired probability is the one that, when taking uniformly at random a permutation of $K_n^\vee$ elements, the $m$ first elements of the permutation belong to a fixed subset of $K_n^\wedge$ elements.
Thus, we infer that 
\[
\proba\pp{E_n\topp m\cap B_n\topp m\cap T_n\topp m} =
\esp\bb{
\frac{K_n^\wedge(K_n^\wedge-1)\cdots( K_n^\wedge-m+1)}{K^\vee_n(K^\vee_n-1)\cdots(K^\vee_n-m+1)}
\ind_{B_n\topp m}\ind_{T_n\topp m}}.
\]
The quotient in the expectation converges to $1$ almost surely and it is bounded by $1$. Therefore, by the dominated convergence theorem, we obtain \eqref{eq:E_n^m} and thus $\limn\proba(E_n\topp m)=1$.

From now on, we assume that the events $E_n\topp m$ and $B_n\topp m$ hold. 
Let $k\in\{1,\ldots,m\}$ be fixed and denote $\ell_k=\what\ell_{n,k}=\wt\ell_{n,k}$.
Recall our definition of $\tau_i$, the $i$-th arrival time of the Poisson process $N$ in the Poissonization and set 
\[
\rho_n:=\max_{i=1,\dots,n}|i-\tau_i|,
\]
the maximal displacement of the positions $1,\dots,n$ by the Poissonization. Consider also the Poisson process $N_k$ derived from $N$ by keeping only the arrival times corresponding to the box $\ell_k$ ($N_k(t):=\summ i1\infty \indd{\tau_i \le t}\indd{Y_i =\ell_k}, t\ge0$). Thus, $N_k$ is a Poisson process of intensity $p_{\ell_k}$ and we denote by $\tau^{(k)}_1< \tau^{(k)}_2< \cdots$ its consecutive arrival times.

Consider $i\in \what R_{n,k}$ and first assume that $i$ is such that $\tau_i\le n$ and hence $\tau_i\in \wt R_{n,k}$. 
In this case we have $d(i,\wt R_{n,k})\le |i-\tau_i|\le \rho_n$. On the other hand, for $i\in \what R_{n,k}$ such that $\tau_i> n$, we have 
\[
d(i,\wt R_{n,k})\le |i-\tau^{(k)}_{N_k(n)}|\le |i-n|\wedge|\tau^{(k)}_{N_k(n)}-n|.
\]
Since in this case $N(n)<i<n$, we have $|i-n|\le |N(n)-n|$ and hence
\equh\label{eq:dH1} 
\sup_{i\in\what R_{n,k}} d(i,\wt R_{n,k})\le \max\left\{\rho_n, |N(n)-n|,|\tau^{(k)}_{N_k(n)}-n| \right\}.
\eque
Now, consider $\tau_i\in\wt R_{n,k}$. For such $\tau_i$ with $i\in\{1,\dots,n\}$, we have  $d(\tau_i,\what R_{n,k})\le |\tau_i-i|\le \rho_n$, whereas for $\tau_i\in\wt R_{n,k}$ with $i>n$, denoting by $j_k$ the maximum of $\what R_{n,k}$ (non-empty by assumption), we have
\[
 d(\tau_i,\what R_{n,k})\le |\tau_i-j_k|\le |\tau_i-\tau_{j_k}|+|\tau_{j_k}-j_k|\le  |n-\tau_{j_k}|+|\tau_{j_k}-j_k|,
\]
where we used that $\tau_{j_k}\le \tau_i\le n$ in the last inequality.
Note that $\tau_{j_k}=\tau^{(k)}_{N_k(\tau_n)}$ and thus,
\[
  \sup_{i :\, \tau_i\in  \wt R_{n,k}
  } d(\tau_i,\what R_{n,k})\le \rho_n + |\tau^{(k)}_{N_k(\tau_n)}-n|.
\]
Therefore, above and \eqref{eq:dH1} yield
\[
 \dH\pp{\frac{\what R_{n,k}}n,\frac{\wt R_{n,k}}n} \le \max\left\{\frac{|N(n)-n|}{n},\frac{|\tau^{(k)}_{N_k(n)}-n|}{n}, \frac{\rho_n}{n} + \frac{|\tau^{(k)}_{N_k(\tau_n)}-n|}{n}  \right\}.
\]

It is well known that $\limn \rho_n/n=0$ and $\limn{|N(n)-n|}/{n} = 0$ almost surely. Furthermore,
\[
\limn\frac{\tau^{(k)}_{N_k(n)}}{n}=\limn\frac{\tau^{(k)}_{N_k(n)}}{N_k(n)}\frac{N_k(n)}{n} = p_{\ell_k}\frac1{p_{\ell_k}}=1 \mbox{  almost surely} 
\]
 and hence  $\limn{\tau^{(k)}_{N_k(\tau_n)}}/{n}= 1$ almost surely.
This established \eqref{eq:dH_approximation} and the proposition.
\end{proof}

\begin{Lem}\label{lem:tildeKn}
Let $\{\wt K_n\}_{n\ge 1}$ be the process defined in \eqref{eq:def-tildeKn} and $\wt\Phi_n=\esp \wt K_n$, $n\ge1$. For any real constant $c$, we have
\[
 \frac{\wt\Phi_n}{(\wt K_n+c)\vee1}\to 1,\; \text{ as }n\to \infty,
\]
almost surely and in $L^p$ for all $p\ge 1$.
\end{Lem}
\begin{proof}
We know from \citep{gnedin07notes} that $\wt K_n\sim\wt\Phi_n$ almost surely and thus the almost sure convergence above follows.
Recalling that $\wt K_n$ is a sum of independent $\{0,1\}$-valued random variables and that $\var(\wt K_n)=\wt\Phi_{2n}-\wt\Phi_n\le \wt\Phi_n$, the Bernstein inequality (see e.g.~\citep{boucheron13concentration}) gives
\begin{align*}
 \proba\pp{\left|\frac{\wt K_n}{\wt \Phi_n}-1\right|>\frac12}&\le 2\exp\pp{-\frac{(\wt\Phi_n/2)^2}{2(\var(\wt K_n)+\wt\Phi_n/6)}}\le 2\exp\pp{-\frac{3}{28}\wt\Phi_n}.
\end{align*}
Let $p\ge 1$ and $q>p$ be fixed. Using the above inequality and the fact that 
$\wt K_n/((\wt K_n+c)\vee1) \le 1\vee(1-c)$, we have
\begin{align*}
\esp\pp{\frac{\wt \Phi_n}{(\wt K_n+c)\vee1}}^q
&=  \esp\pp{\pp{\frac{\wt \Phi_n}{(\wt K_n+c)\vee1}}^q\ind_{\left\{\frac{\wt K_n}{\wt\Phi_n}\ge\frac12\right\}} } +  \esp\pp{\pp{\frac{\wt \Phi_n}{(\wt K_n+c)\vee1}}^q\ind_{\left\{\frac{\wt K_n}{\wt\Phi_n}<\frac12\right\}} }\\
&\le 2^q (1\vee(1-c))^q + 2 \wt\Phi_n^q \exp\pp{-\frac{3}{28}\wt\Phi_n}.
\end{align*}
We infer that $\{{\wt \Phi_n}/((\wt K_n+c)\vee1)\}_{n\ge 1}$ is bounded in $L^q$ and then $\{[\wt \Phi_n/((\wt K_n+c)\vee1)]^p\}_{n\ge 1}$ is uniformly integrable. The desired $L^p$ convergence follows.
\end{proof}

\begin{proof}[Proof of Theorem \ref{thm:1}]
To prove the convergence of the point processes of interest, we compute their Laplace transform:
\[
\Psi_{\xi_n}(f) := \esp\exp \pp{-\xi_n(f)} = \esp\exp\pp{-\summ k1{K_n}f\pp{\varepsilon_{n,k}/b_n,\what R_{n,k}/n}},
\]
for $f\in C_K^+((0,\infty)\times\calF([0,1]))$, the space of non-negative continuous functions with compact support. Similarly,
\[
\Psi_\xi(f) := \esp \exp\pp{-\sif \ell1 f\pp{\Gamma_\ell^{-1/\alpha},\calR_\ell^{(\beta)}}}
\]
is the Laplace transform of $\xi$. 
Recall that the desired convergence follows if and only if 
\equh\label{eq:Laplace_convergence}
\limn \Psi_{\xi_n}(f) = \Psi_\xi(f), \mfa f\in C_K^+((0,\infty),\calF([0,1])),
\eque
see for example \citep[Proposition 3.19]{resnick87extreme}.

Now we prove \eqref{eq:Laplace_convergence}. When investigating the weak convergence of point processes here, the topology on $(0,\infty)$ is such that all compact sets are bounded away from zero and $\calF([0,1])$ is itself a compact metric space. So, for any $f\in C_K^+((0,\infty)\times\calF([0,1]))$,  there exists $\kappa=\kappa(f)>0$ so that $f(x,F) = 0$ for all $x<\kappa$ and $F\in\calF([0,1])$. Given $f$ and thus $\kappa>0$ fixed, for all $\epsilon>0$, we can pick $m=m(\kappa,\epsilon)\in\N$ large enough, so that 
\[
\limn\proba\pp{B\topp m_{\kappa,n}} = \proba\pp{\Gamma_m^{-1/\alpha}<\kappa}>1-\epsilon \quad \text{ with }  B\topp m_{\kappa,n} := \ccbb{\frac{\varepsilon_{n,m}}{b_n}<\kappa}.
\]
Now we express $\Psi_{\xi_n}(f)$ as
\[
\Psi_{\xi_n}(f) = \esp\bb{\exp\pp{-\xi_n(f)}\ind_{B\topp m_{\kappa,n}}} + \esp\bb{\exp\pp{-\xi_n(f)}\ind_{(B\topp m_{\kappa,n})^c}}.
\]
The second term on the right-hand side above is then bounded by $1-\proba(B_{\kappa,\epsilon,n}\topp m)$. The first term equals
\equh\label{eq:CMT}
\esp\bb{\exp\pp{-\summ k1{m}f\pp{\varepsilon_{n,k}/b_n,\what R_{n,k}/n}}\ind_{B\topp m_{\kappa,n}}}.
\eque
This is the expectation of a function from $\R_+^m\times \calF([0,1])^m$ to $[0,1]$, continuous everywhere except at points from the set 
\equh\label{eq:discontinuity}
\ccbb{(x_1,\dots,x_m,F_1,\dots,F_m)\in\R_+^m\times\calF([0,1])^m: x_m = \kappa}.
\eque
We have seen the convergence $(\varepsilon_{n,k}/b_n,\what R_{n,k}/n)_{k=1,\dots,m}\weakto(\Gamma_k^{-1/\alpha},\calR_k\topp\beta)_{k=1,\dots,m}$  in Proposition~\ref{prop:m}, and we can notice that the set of discontinuity points \eqref{eq:discontinuity} above is hit by $(\Gamma_1^{-1/\alpha},\dots,\Gamma_m^{-1/\alpha},\calR_1\topp\beta,\dots,\calR_m\topp\beta)$ with probability zero. Therefore, applying the continuous mapping theorem to \eqref{eq:CMT}, we have that
\begin{align*}
\limsupn \Psi_{\xi_n}(f) & \le \esp\bb{\exp\pp{-\summ k1{m}f\pp{\Gamma_k^{-1/\alpha},\calR\topp\beta_k}}\inddd{\Gamma_m^{-1/\alpha}<\kappa}} + \epsilon\\
& = \esp\bb{\exp\pp{-\summ k1{\infty}f\pp{\Gamma_k^{-1/\alpha},\calR\topp\beta_k}}\inddd{\Gamma_m^{-1/\alpha}<\kappa}} + \epsilon\\
&\le \Psi_\xi(f)+ \epsilon.
\end{align*}
Similarly, one can show that
\[
\liminfn\Psi_{\xi_n}(f) \ge \Psi_\xi(f)-\proba\pp{\Gamma_m^{-1/\alpha}\ge \kappa} \ge \Psi_\xi(f)-\epsilon. 
\] 
Since $\epsilon>0$ is arbitrary, we have thus proved \eqref{eq:Laplace_convergence} for every test function $f$, and hence the desired result.
\end{proof}

\begin{proof}[Proof of Theorem \ref{thm:2}]
It suffices to prove, for all open intervals $A_1,\dots,A_d$ in $[0,1]$ and positive reals $z_1,\dots,z_d$, that
\begin{align*}
P_n &:=\proba\pp{\frac{M_n(A_1)}{b_n}> z_1,\dots,\frac{M_n(A_d)}{b_n}> z_d} \\
&\longrightarrow \, \proba(\mab(A_1)> z_1,\dots,\mab(A_d)> z_d):=P, \mmas n\to\infty.
\end{align*}
This is a direct consequence of Theorem~\ref{thm:1} since, denoting 
\[
\calF_{A_i}=\{F\in\calF([0,1]):F\cap A_i \ne \emptyset\},\; i=1,\ldots,d,
\]
we have
\begin{align*}
P_n
&=  \proba\pp{\xi_n\pp{(z_1,\infty)\times \calF_{A_1}}\ge 1,\dots,\xi_n\pp{(z_d,\infty)\times \calF_{A_d}}\ge 1} \\
& \longrightarrow \,  \proba\pp{\xi\pp{(z_1,\infty)\times \calF_{A_1}}\ge 1,\dots,\xi\pp{(z_d,\infty)\times \calF_{A_d}}\ge 1}=P, \mmas n\to\infty.
\end{align*}
\end{proof}

Our proof of Theorem~\ref{thm:2} is based on the presentation \eqref{eq:Karlin_rsm01} of $\mab$, which we have shown at the beginning can be derived from the presentation \eqref{eq:fdd}. 
We conclude this section by giving a direct proof of Theorem \ref{thm:2} using the presentation \eqref{eq:fdd} and also without using Proposition \ref{prop:m}. 
\begin{proof}[Second proof of Theorem \ref{thm:2}]
Fix $d\in\N$, open intervals $A_1,\dots,A_d$ in $[0,1]$ and positive reals $z_1,\dots,z_d$. We shall prove that 
\[
\proba\pp{\frac{M_n(A_k)}{b_n}\le z_k, \;k=1,\ldots,d}\to \proba\pp{\mab(A_k)\le z_k,\;k=1,\ldots,d},
\]
as $n\to\infty$.
For every $\ell\in\N$ and every $n\in\N$, we record whether $Y_i = \ell$ for some $i\in nA_k$, for each $k=1,\dots,d$, and count different types of boxes. More precisely, introduce $\vvdelta = (\delta_1,\dots,\delta_d)\in\Lambda_d :=\{0,1\}^d\setminus\{0,\dots,0\}$, and consider
\[
\tau_{\vvA}^\vvdelta(n):= \sif \ell1 \prod_{\substack{k=1,\dots,d\\\delta_k = 1}}\inddd{\exists i\in nA_k, Y_i = \ell}\prod_{\substack{k'=1,\dots,d\\\delta_{k'} = 0}}\inddd{\forall i\in nA_{k'}, Y_i\ne \ell}.
\]
For example, $\tau_\vvA^{1,\dots,1}(n)$ is the number of box $\ell$ that has been sampled in some round $i_1\in nA_1, i_2\in nA_2,\dots,i_d\in nA_d$, and $\tau_\vvA^{1,0,\dots,0}(n)$ is the number of box $\ell$ that has been sampled in some round $i_1\in nA_1$, but never in any round in $nA_2,\dots,nA_d$. So all boxes that have been sampled during the first $n$ rounds are divided into disjoint groups indexed by $\vvdelta\in\Lambda_d$.

Now we need the following limit theorem for $\tau_\vvA^\vvdelta(n)$:
\equh\label{eq:tau}
\limn\frac{\tau_\vvA^\vvdelta(n)}{n^\beta L(n)} = \tau_\vvA^\vvdelta:= \int_0^\infty \beta x^{-\beta-1}\wt\proba\pp{\inddd{\wt\calN(xA_k) \ne 0} = \delta_k, k=1,\dots,d}dx
\eque
in probability. This follows from \citep[Theorem~2]{durieu17infinite} (which was also established by the Poissonization technique): the above identity therein was established for the corresponding Poisson random measures being even or odd, and we obtain the desired result here by applying the identity 
\[
\wt\proba\pp{\wt\calN(A)\ne\emptyset} = \frac12\wt\proba\pp{\wt\calN(2A)\ \rm odd}.
\]
Then, conditioning on $\indn Y$, we can write
\begin{align*}
\proba &\pp{\frac{M_n(A_k)}{b_n} \le z_k,k=1,\dots,d}\\
& = \esp  \bb{\prod_{\vvdelta\in\Lambda_d}\proba_0\pp{\frac{\varepsilon_0}{b_n}\le\min_{k=1,\dots,d, \delta_k = 1}z_k}^{\tau_\vvA^\vvdelta(n)}} \\
& = \esp  \exp\ccbb{ \sum_{\vvdelta\in\Lambda_d}\tau_\vvA^\vvdelta(n)\log\bb{1-\proba_0\pp{\frac{\varepsilon_0}{b_n}>\min_{k=1,\dots,d, \delta_k = 1}z_k}}},\\
\end{align*}
where $\varepsilon_0$, defined on another probability space $(\Omega_0, \calF_0,\P_0)$, has the same distribution as $\varepsilon_1$.
By \eqref{eq:tau} and heavy-tail assumption \eqref{eq:RVepsilon} on $\varepsilon$'s distribution, 
\begin{align*}
 \limn &\sum_{\vvdelta\in\Lambda_d}\tau_\vvA^\vvdelta(n)\log\bb{1-\proba_0\pp{\frac{\varepsilon_0}{b_n}>\min_{k=1,\dots,d, \delta_k = 1}z_k}}\\
&=\Gamma(1-\beta)^{-1}\sum_{\vvdelta\in\Lambda_d}\tau_\vvA^\vvdelta\, \pp{\min_{k=1,\dots,d, \delta_k = 1}z_k}^{-\alpha}\;\text{in probability.}
\end{align*}
This last sum can be written as
\begin{align*}
\wt\esp & \pp{\sum_{\vvdelta\in\Lambda_d}\max_{k=1,\dots,d,\delta _k = 1}\frac1{z_k^\alpha}\prodd k1d \inddd{\indd{\wt \calN(xA_k) \ne 0} = \delta_k}}
\\
& = \wt\esp\pp{\sum_{\vvdelta\in\Lambda_d}\max_{k=1,\dots,d}\frac{\indd{\wt\calN(xA_k) \ne0}}{z_k^\alpha}\inddd{\inddd{\wt\calN(xA_k)\ne 0} = \delta_k, k=1,\dots,d}}\\
& = \wt\esp \pp{\max_{k=1,\dots,d}\frac{\indd{\wt\calN(xA_k) \ne0}}{z_k^\alpha}}.
\end{align*}
Summing up, we have thus shown that 
\begin{align*}
\limn\proba & \pp{\frac{M_n(A_k)}{b_n}\le z_k,\;k=1,\dots,d}\\
&= \exp\pp{-\Gamma(1-\beta)^{-1}\int_0^\infty \beta x^{-\beta-1} \wt\esp\pp{\max_{k=1,\dots,d}\frac{\indd{\wt N(xA_k)\ne0}}{z_k^\alpha}}dx},
\end{align*}
which is the desired finite-dimensional distribution as in \eqref{eq:fdd}.
\end{proof}

\section{Random sup-measures and associated extremal processes}
\label{sec:discussions}

The extremal process associated to the Karlin random sup-measure $\calM_{\alpha,\beta}$ appears to be a time-changed version of a standard $\alpha$-Fr\'echet extremal process $\M_\alpha$, precisely 
\equh\label{eq:extremal0}
\ccbb{\M_\alpha(t^\beta)}_{t\ge0}.
\eque
As noticed in Section~\ref{sec:choquet}, this is a consequence of the more general fact that the extremal process of any Choquet $\alpha$-Fr\'echet random sup-measure is determined by the extremal coefficient functional evaluated on sets $\{[0,t]\}_{t>0}$ only.
This is proved in Proposition~\ref{prop:extremal} in the appendix.
The Karlin random sup-measure is of course not the only Choquet random sup-measure corresponding to the same extremal process \eqref{eq:extremal0}.
Another such family that arises naturally from limit theorems with long-range dependence are the stable-regenerative random sup-measures \citep{lacaux16time} recalled below.
\begin{Example}\label{Rem:LS}
We recall the definition of stable-regenerative random sup-measures:
\equh\label{eq:LS}
\calM^{\rm sr}_{\alpha,\beta}(\cdot):=\sup_{\ell\in\N}\frac1{\Gamma_\ell^{1/\alpha}}\inddd{\pp{V_\ell\topp\beta+R\topp\beta_\ell
}\cap\,\cdot\,\ne\emptyset},
\eque
where  $\{(\Gamma_\ell,V\topp\beta_\ell,R\topp\beta_\ell)\}_{\ell\in\N}$
is a Poisson point process on $\R_+\times\R_+\times\calF(\R_+)$ with intensity 
$dx\beta v^{-(1-\beta)}dvdP_{1-\beta}$ where $P_{1-\beta}$
is the law of $(1-\beta)$-stable regenerative set (i.e., the closure of a $(1-\beta)$-stable
subordinator \citep{bertoin99subordinators}) on $\R_+$, and $\wt R_\ell\topp\beta$ in \eqref{eq:LS0} is $V\topp\beta_\ell + R\topp\beta_\ell$ here.  It was shown \citep{lacaux16time,owada15maxima} that
\[
\ccbb{\calM_{\alpha,\beta}^{\rm sr}([0,t])}_{t\ge 0} \eqd  \ccbb{\M_\alpha(t^{\beta})}_{t\ge 0}.
\]
(Strictly speaking only $\beta\in(1/2,1)$
 was considered in \citep{lacaux16time}, although the extension to $\beta\in(0,1)$ is straightforward.)
\end{Example}

We now give an example of random sup-measure that is self-similar, {\em non-stationary}, and yet also has the same extremal process. 
\begin{Example}
For $\beta>0$, let $T_\beta$ be the mapping between subsets of $\R_+$ induced by $t\mapsto t^\beta$. Then, $\calM_\alpha\circ T_\beta$ is $\beta/\alpha$-self-similar, but non-stationary, and the corresponding extremal process also has the form $\{\M_\alpha(t^\beta)\}_{t\ge0}$.
\end{Example}
In the special case $\beta\in(0,1)$, we provide another equivalent representation of $\calM_\alpha\circ T_\beta$, which can also be connected to a variation of the Karlin model investigated in Section~\ref{sec:Karlin}.
Let $\wt\calN$ be a Poisson random measure on $\R_+$, and view it as a Poisson process by letting $\wt\calN(t)=\wt\calN([0,t])\in\N_0:=\{0\}\cup\N$ denote the counting number of the Poisson process. We write 
\[
\wt\calN[A] := \ccbb{\wt\calN(t):t\in A} \subset\N_0, 
\mbox{ for } A\subset\R_+.
\] 
We then introduce
\equh\label{eq:occupancy}
\calM^*_{\alpha,\beta}(\cdot) := \sup_{\ell\in\N}\frac1{\Gamma_\ell^{1/\alpha}}\inddd{\wt N_\ell\bb{x_\ell\,\cdot\,}\ni 1} \mbox{ on } \R_+,
\eque
with $\{(\Gamma_\ell, \wt\calN_\ell, x_\ell)\}_{\ell\in\N}$ defined as in \eqref{eq:Karlin_rsm}.
When restricted to $[0,1]$, 
\[
\calM^*_{\alpha,\beta} (\cdot) \eqd  \sup_{\ell\in\N}\frac1{\Gamma_\ell^{1/\alpha}}\inddd{\calR^{(\beta)*}_\ell\cap\,\cdot\,\ne\emptyset},
\]
with $\calR^{(\beta)*} \eqd \min\calR\topp\beta$ (recall \eqref{eq:R}). In fact, one could define $\mab$ and $\mab^*$ based on the same Poisson point process such that with probability one, $\mab(\cdot)\ge \mab^*(\cdot)$.
\begin{Prop}Let $\calM_\alpha$ be defined as in \eqref{eq:completely_random} and $T_\beta$ be the mapping between subsets of $\R_+$ induced by $t\mapsto t^\beta$ for some $\beta\in(0,1)$, then 
\equh\label{eq:shifted}
\calM_{\alpha,\beta}^* \eqd \calM_\alpha\circ T_\beta
\eque
as random sup-measures on $\R_+$. 
\end{Prop} 
\begin{proof}
To show \eqref{eq:shifted}, by self-similarity it suffices to restrict to $[0,1]$ and compare the capacity functionals of the random closed sets in the Poisson point process presentation \eqref{eq:occupancy} and \eqref{eq:completely_random}. 
We start by computing the extremal coefficient functional corresponding to \eqref{eq:occupancy}: for an interval $A = (a,b]$,
\begin{align*}
\nonumber \Gamma(1-\beta)^{-1}\int_0^\infty  & \beta x^{-\beta-1} 
 \wt\proba\pp{\wt \calN\bb{xA}\ni 1}dx\\
\nonumber& = \Gamma(1-\beta)^{-1}\int_0^\infty \beta x^{-\beta-1}\wt\proba\pp{\wt \calN(xa) = 0, \wt\calN(xb)>0} dx\\
\nonumber& = \Gamma(1-\beta)^{-1}\int_0^\infty\beta x^{-\beta-1}\bb{\wt\proba\pp{\wt \calN(xa) = 0} - \wt\proba\pp{\wt\calN(xb) = 0}}dx \\
& = (b^\beta - a^\beta). 
\end{align*}
This implies that the capacity functional for $\calR^{(\beta)*}$ is 
\[
\theta([a,b]) \proba\pp{\calR^{(\beta)*}\cap [a,b]\ne\emptyset} = b^\beta-a^\beta = \Leb(T_\beta([a,b])),
\] whence
\[
\calM_{\alpha,\beta}^*(\cdot) \eqd \sup_{\ell\in\N}\frac1{\Gamma_\ell^{1/\alpha}}\inddd{U_\ell\in T^\beta(\cdot)},
\]
with $\indn U$ being i.i.d.~uniform random variables on $[0,1]$. 
The desired result hence follows. 
\end{proof}

The above representation of $\calM_\alpha\circ T_\beta$ was discovered during our investigation on the limit of empirical random sup-measures for the following variation of the Karlin model
\[
X_n^*:=\varepsilon_{Y_n}\inddd{K_{n,Y_n} = 1}, \quad n\in\N,
\]
with $\indn Y$, $\indn\varepsilon$ and $K_{n,\ell}$, as in Section \ref{sec:Karlin}. In this variation, 
if a box $\ell$ is sampled  ($Y_n = \ell$), then $X_n^* = \varepsilon_\ell$ only if this is the first time for the box $\ell$, and $X_n^* = 0$ otherwise. 
For this model, one could establish a limit theorem for the empirical random sup-measure, and the limit is exactly the random sup-measure $\calM^*_{\alpha,\beta}$.
The sequence $\{X_n^*\}_{n\in\N}$ is not stationary, a drastically difference from $\{X_n\}_{n\in\N}$ considered in Section \ref{sec:Karlin}. Nevertheless, we see that partial maxima of both sequences are equal, explaining the equality of the corresponding extremal processes in the limit.

We conclude this section by the following remark comparing the aforementioned random sup-measures.

\begin{Rem}
In summary, for $\beta\in(0,1)$,
\[
\calM_{\alpha,\beta},\, \calM_{\alpha,\beta}^{\rm sr}, \,\text{and } \calM_\alpha\circ T_\beta
\]all have the same extremal process as $\{\M_\alpha(t^\beta)\}_{t\ge 0}$. 
The independently scattered random sup-measure $\calM_\alpha$, the stable-regenerative random sup-measure $\calM_{\alpha,\beta}^{\rm sr}$, and the Karlin random sup-measure $\calM_{\alpha,\beta}$ are all self-similar and shift-invariant. However, 
for the corresponding max-increment processes \eqref{eq:max-increment}, $\calM_\alpha$ is mixing, $\calM^{\rm sr}_{\alpha,\beta}$ is ergodic but not mixing, and $\calM_{\alpha,\beta}$ is not ergodic. 
 The random sup-measure $\calM_\alpha\circ T_\beta$ is self-similar but not shift-invariant. 
\end{Rem}

\appendix

\section{Extremal processes of Choquet random sup-measures}\label{appendix}
As before, given a random sup-measure $\calM$, we let $\M(t):=\calM([0,t])$, $t\ge0$, denote its associated extremal process. 
We denote by $\M_\alpha$ the standard $\alpha$-Fr\'echet extremal process defined in \eqref{eq:def_ext_pro}.
In the literature, $\M_\alpha$ was originally named {\em the} extremal process \citep{dwass64extremal,lamperti64extreme}. The notion has become however more and more common to refer to various limits of partial-maxima processes. The same notion was also used for random sup-measures in \citep{obrien90stationary}.

Recall the definition Choquet random sup-measures \eqref{eq:Choquet} in Section~\ref{sec:choquet}. 
Proposition \ref{prop:ssChoquet} therein is a special case of the following result.
\begin{Prop}\label{prop:extremal}
Let $\calM$ be a Choquet $\alpha$-Fr\'echet random sup-measure with extremal coefficient functional $\theta$, and $\M$ its extremal process. 

\noindent (i) For 
$d\in\N, 0<t_1<\cdots<t_d$ and $x_1,\dots,x_d\in\R_+$, 
\equh\label{eq:referee}
\proba\pp{\M(t_k)\le x_k, k=1,\dots,d} = \exp\pp{-\summ k1d a_k \theta([0,t_k])} 
\eque
with
\[
 a_k := \frac1{\bweee jkd x_j} - \frac1{\bweee j{k+1}dx_j}, k=1,\dots,d-1,
\]
and $a_d :=1/x_d$. 

\noindent (ii) If in addition $\calM$ is $H$-self-similar with $H>0$, then,
\[
\theta([0,1])\ccbb{\M(t)}_{t\ge 0}\eqd  \ccbb{\M_{\alpha}(t^{\alpha H})}_{t\ge 0}.
\]
\end{Prop}
\begin{proof}
We start by computing the finite-dimensional distribution of the associated extremal process. We write
\begin{align*}
\proba(\M(t_k)\le x_k, k=1,\dots,d) & = \proba\pp{\calM([0,t_k])\le x_k, k=1,\dots,d}  \\
& = \proba\pp{\int_{\R_+}^\vee \bigvee_{k=1}^d\frac{\ind_{[0,t_k]}}{\bigwedge_{j=k}^dx_j}d\calM\le 1}.
\end{align*}
See \citep{stoev06extremal} for background on stochastic extremal integrals $\int^\vee fd\calM$. 
We then express the integrand as
\[
f(t) := \bveee k1d \frac{\ind_{[0,t_k]}(t)}{\bweee jkd x_j} = \summ k1d a_k \ind_{[0,t_k]}(t).
\]
In this way, we see that $f$ is an upper-semi-continuous function expressed as the sum of $d$ comonotonic functions. Let $\theta$ denote the extremal coefficient functional of $\calM$. From \citep{molchanov16max}, 
we know that $\proba(\int^\vee gd\calM\le t) = \exp(-\ell(g)/t)$, $t>0$, where here and below, $\ell(g) := \int gd\theta$ (understood as a Choquet integral for upper-semi-continuous function $g$) is the {\em tail dependence functional} of $\calM$, and $\ell(\ind_K) = \theta(K)$.
In particular we have
\begin{align}
\proba(\M(t_k)\le x_k, k=1,\dots,d) & =
 \proba\pp{\int^\vee fd\calM \le 1}  = \exp\pp{-\ell(f)} \label{eq:TDF}\\
 &
 = \exp\pp{-\summ k1d a_k \theta([0,t_k])},\nonumber
\end{align}
and in the last step we applied the comonotonic additivity of the tail dependence function $\ell$ for Choquet random sup-measures (i.e., for comonotonic functions $g,h$, $\int g+hd\theta = \int gd\theta+\int hd\theta$ \citep{denneberg94nonadditive,molchanov16max}). 
We have proved  the first part of the proposition.

We also know that for an $H$-self-similar $\alpha$-Fr\'echet random sup-measure, the extremal coefficient functional necessarily has the scaling property $\theta(\lambda[0,t]) = \lambda^{\alpha H}\theta([0,t])$ for all $\lambda>0$ (see \eqref{eq:ecf0}). So for such a random sup-measure the conclusion of the first part  becomes
\[
\proba\pp{\M(t_k)\le x_k, k=1,\dots,n} = \exp\pp{-\theta([0,1])\summ k1\ell a_k t_k^{\alpha H}}.
\]
Recall that for the independently scattered random sup-measure $\calM_\alpha$, extremal coefficient functional is the Lebesgue measure.  The second part of the proposition then follows.
\end{proof}
\begin{Rem}We thank an anonymous referee for pointing out to us the following consequence: for a general  Fr\'echet random sup-measure not of Choquet type, the statement \eqref{eq:referee} holds with `$=$' replaced by `$\ge$'. This is due to the stochastic dominance property of Choquet random sup-measures. 
Indeed, a general Fr\'echet random sup-measure $\wt \calM$ can be coupled with a Choquet random sup-measure $\calM$ {\em with the same extremal coefficient functional} $\theta$. 
Let $\wt \M, \M$ and $\wt\ell,\ell$ be the extremal processes and the tail dependence functionals of the two random sup-measures, respectively. 
It is shown in \citep[Corollary 5.4]{molchanov16max} that $\wt\ell\le \ell$. Now, the aforementioned statement follows from the fact that the law of the extremal process is uniquely determined by the  tail dependence functional \eqref{eq:TDF}.
\end{Rem}

\subsection*{Acknowledgments} 
The authors would like to thank Kirstin Strokorb for several helpful comments and for pointing out the discussions of Karlin random sup-measures in the earlier arXiv version of \citep{molchanov16max}.
The authors also thank two anonymous referees for their careful reading of the manuscript and their insightful comments.
YW's research was partially supported by 
the NSA grant H98230-16-1-0322, the ARO grant W911NF-17-1-0006, and Charles Phelps Taft Research Center 
at University of Cincinnati.

\bibliographystyle{apalike}
\bibliography{references}

\def\cprime{$'$} \def\polhk#1{\setbox0=\hbox{#1}{\ooalign{\hidewidth
  \lower1.5ex\hbox{`}\hidewidth\crcr\unhbox0}}}
  \def\polhk#1{\setbox0=\hbox{#1}{\ooalign{\hidewidth
  \lower1.5ex\hbox{`}\hidewidth\crcr\unhbox0}}}
\begin{thebibliography}{}

\bibitem[Beran et~al., 2013]{beran13long}
Beran, J., Feng, Y., Ghosh, S., and Kulik, R. (2013).
\newblock {\em Long-memory processes}.
\newblock Springer, Heidelberg.
\newblock Probabilistic properties and statistical methods.

\bibitem[Bertoin, 1999]{bertoin99subordinators}
Bertoin, J. (1999).
\newblock Subordinators: examples and applications.
\newblock In {\em Lectures on probability theory and statistics
  ({S}aint-{F}lour, 1997)}, volume 1717 of {\em Lecture Notes in Math.}, pages
  1--91. Springer, Berlin.

\bibitem[Billingsley, 1999]{billingsley99convergence}
Billingsley, P. (1999).
\newblock {\em Convergence of probability measures}.
\newblock Wiley Series in Probability and Statistics: Probability and
  Statistics. John Wiley \& Sons Inc., New York, second edition.
\newblock A Wiley-Interscience Publication.

\bibitem[Bingham et~al., 1987]{bingham87regular}
Bingham, N.~H., Goldie, C.~M., and Teugels, J.~L. (1987).
\newblock {\em Regular variation}, volume~27 of {\em Encyclopedia of
  Mathematics and its Applications}.
\newblock Cambridge University Press, Cambridge.

\bibitem[Boucheron et~al., 2013]{boucheron13concentration}
Boucheron, S., Lugosi, G., and Massart, P. (2013).
\newblock {\em Concentration inequalities}.
\newblock Oxford University Press, Oxford.
\newblock A nonasymptotic theory of independence, With a foreword by Michel
  Ledoux.

\bibitem[de~Haan, 1984]{dehaan84spectral}
de~Haan, L. (1984).
\newblock A spectral representation for max-stable processes.
\newblock {\em Ann. Probab.}, 12(4):1194--1204.

\bibitem[de~Haan and Ferreira, 2006]{dehaan06extreme}
de~Haan, L. and Ferreira, A. (2006).
\newblock {\em Extreme value theory}.
\newblock Springer Series in Operations Research and Financial Engineering.
  Springer, New York.
\newblock An introduction.

\bibitem[Denneberg, 1994]{denneberg94nonadditive}
Denneberg, D. (1994).
\newblock {\em Non-additive measure and integral}, volume~27 of {\em Theory and
  Decision Library. Series B: Mathematical and Statistical Methods}.
\newblock Kluwer Academic Publishers Group, Dordrecht.

\bibitem[Dombry and Kabluchko, 2017]{dombry17ergodic}
Dombry, C. and Kabluchko, Z. (2017).
\newblock Ergodic decompositions of stationary max-stable processes in terms of
  their spectral functions.
\newblock {\em Stochastic Process. Appl.}, 127(6):1763--1784.

\bibitem[Durieu et~al., 2017]{durieu17infinite}
Durieu, O., Samorodnitsky, G., and Wang, Y. (2017).
\newblock From infinite urn schemes to self-similar stable processes.
\newblock Submitted, available at \url{https://arxiv.org/abs/1710.08058}.

\bibitem[Durieu and Wang, 2016]{durieu16infinite}
Durieu, O. and Wang, Y. (2016).
\newblock From infinite urn schemes to decompositions of self-similar
  {G}aussian processes.
\newblock {\em Electron. J. Probab.}, 21:Paper No. 43, 23.

\bibitem[Dwass, 1964]{dwass64extremal}
Dwass, M. (1964).
\newblock Extremal processes.
\newblock {\em Ann. Math. Statist}, 35:1718--1725.

\bibitem[Gnedin et~al., 2007]{gnedin07notes}
Gnedin, A., Hansen, B., and Pitman, J. (2007).
\newblock Notes on the occupancy problem with infinitely many boxes: general
  asymptotics and power laws.
\newblock {\em Probab. Surv.}, 4:146--171.

\bibitem[Hsing et~al., 1988]{hsing98exceedance}
Hsing, T., H\"usler, J., and Leadbetter, M.~R. (1988).
\newblock On the exceedance point process for a stationary sequence.
\newblock {\em Probab. Theory Related Fields}, 78(1):97--112.

\bibitem[Kabluchko, 2009]{kabluchko09spectral}
Kabluchko, Z. (2009).
\newblock Spectral representations of sum- and max-stable processes.
\newblock {\em Extremes}, 12(4):401--424.

\bibitem[Kabluchko and Schlather, 2010]{kabluchko10ergodic}
Kabluchko, Z. and Schlather, M. (2010).
\newblock Ergodic properties of max-infinitely divisible processes.
\newblock {\em Stochastic Process. Appl.}, 120(3):281--295.

\bibitem[Karlin, 1967]{karlin67central}
Karlin, S. (1967).
\newblock Central limit theorems for certain infinite urn schemes.
\newblock {\em J. Math. Mech.}, 17:373--401.

\bibitem[Lacaux and Samorodnitsky, 2016]{lacaux16time}
Lacaux, C. and Samorodnitsky, G. (2016).
\newblock Time-changed extremal process as a random sup measure.
\newblock {\em Bernoulli}, 22(4):1979--2000.

\bibitem[Lamperti, 1962]{lamperti62semi}
Lamperti, J. (1962).
\newblock Semi-stable stochastic processes.
\newblock {\em Trans. Amer. Math. Soc.}, 104:62--78.

\bibitem[Lamperti, 1964]{lamperti64extreme}
Lamperti, J. (1964).
\newblock On extreme order statistics.
\newblock {\em Ann. Math. Statist}, 35:1726--1737.

\bibitem[Leadbetter et~al., 1983]{leadbetter83extremes}
Leadbetter, M.~R., Lindgren, G., and Rootz{\'e}n, H. (1983).
\newblock {\em Extremes and related properties of random sequences and
  processes}.
\newblock Springer Series in Statistics. Springer-Verlag, New York.

\bibitem[LePage et~al., 1981]{lepage81convergence}
LePage, R., Woodroofe, M., and Zinn, J. (1981).
\newblock Convergence to a stable distribution via order statistics.
\newblock {\em Ann. Probab.}, 9(4):624--632.

\bibitem[Molchanov, 2005]{molchanov05theory}
Molchanov, I. (2005).
\newblock {\em Theory of random sets}.
\newblock Probability and its Applications (New York). Springer-Verlag London,
  Ltd., London.

\bibitem[Molchanov and Strokorb, 2016]{molchanov16max}
Molchanov, I. and Strokorb, K. (2016).
\newblock Max-stable random sup-measures with comonotonic tail dependence.
\newblock {\em Stochastic Process. Appl.}, 126(9):2835--2859.

\bibitem[O'Brien et~al., 1990]{obrien90stationary}
O'Brien, G.~L., Torfs, P. J. J.~F., and Vervaat, W. (1990).
\newblock Stationary self-similar extremal processes.
\newblock {\em Probab. Theory Related Fields}, 87(1):97--119.

\bibitem[Owada and Samorodnitsky, 2015]{owada15maxima}
Owada, T. and Samorodnitsky, G. (2015).
\newblock Maxima of long memory stationary symmetric {$\alpha$}-stable
  processes, and self-similar processes with stationary max-increments.
\newblock {\em Bernoulli}, 21(3):1575--1599.

\bibitem[Pipiras and Taqqu, 2017]{pipiras17long}
Pipiras, V. and Taqqu, M.~S. (2017).
\newblock {\em Long-range dependence and self-similarity}.
\newblock Cambridge Series in Statistical and Probabilistic Mathematics, [45].
  Cambridge University Press, Cambridge.

\bibitem[Pitman, 2006]{pitman06combinatorial}
Pitman, J. (2006).
\newblock {\em Combinatorial stochastic processes}, volume 1875 of {\em Lecture
  Notes in Mathematics}.
\newblock Springer-Verlag, Berlin.
\newblock Lectures from the 32nd Summer School on Probability Theory held in
  Saint-Flour, July 7--24, 2002, With a foreword by Jean Picard.

\bibitem[Resnick, 1987]{resnick87extreme}
Resnick, S.~I. (1987).
\newblock {\em Extreme values, regular variation, and point processes},
  volume~4 of {\em Applied Probability. A Series of the Applied Probability
  Trust}.
\newblock Springer-Verlag, New York.

\bibitem[Samorodnitsky, 2004]{samorodnitsky04extreme}
Samorodnitsky, G. (2004).
\newblock Extreme value theory, ergodic theory and the boundary between short
  memory and long memory for stationary stable processes.
\newblock {\em Ann. Probab.}, 32(2):1438--1468.

\bibitem[Samorodnitsky, 2016]{samorodnitsky16stochastic}
Samorodnitsky, G. (2016).
\newblock {\em Stochastic processes and long range dependence}.
\newblock Springer Series in Operations Research and Financial Engineering.
  Springer, Cham.

\bibitem[Samorodnitsky and Wang, 2017]{samorodnitsky17extremal}
Samorodnitsky, G. and Wang, Y. (2017).
\newblock Extremal theory for long range dependent infinitely divisible
  processes.
\newblock Submitted, available at \url{https://arxiv.org/abs/1703.07496}.

\bibitem[Stoev, 2008]{stoev08ergodicity}
Stoev, S.~A. (2008).
\newblock On the ergodicity and mixing of max-stable processes.
\newblock {\em Stochastic Process. Appl.}, 118(9):1679--1705.

\bibitem[Stoev and Taqqu, 2005]{stoev06extremal}
Stoev, S.~A. and Taqqu, M.~S. (2005).
\newblock Extremal stochastic integrals: a parallel between max-stable
  processes and {$\alpha$}-stable processes.
\newblock {\em Extremes}, 8(4):237--266 (2006).

\bibitem[Vervaat, 1997]{vervaat97random}
Vervaat, W. (1997).
\newblock Random upper semicontinuous functions and extremal processes.
\newblock In {\em Probability and lattices}, volume 110 of {\em CWI Tract},
  pages 1--56. Math. Centrum, Centrum Wisk. Inform., Amsterdam.

\end{thebibliography}
\end{document}